\documentclass[oneside,a4paper,12pt,reqno]{amsart}
\usepackage{typearea}
\typearea{11}

\usepackage[utf8]{inputenc}
\usepackage[UKenglish]{babel}
\usepackage[T1]{fontenc}

\usepackage[shortlabels]{enumitem}

\usepackage{dsfont}
\usepackage{mathtools}
\usepackage{mleftright}

\usepackage{microtype}
\usepackage[colorlinks=false]{hyperref}

\usepackage{tikz}
\usetikzlibrary{calc,backgrounds}

\mathtoolsset{centercolon}
\numberwithin{equation}{section}
\numberwithin{figure}{section}

\theoremstyle{plain}
\newtheorem{thm}{Theorem}[section]
\newtheorem{prop}[thm]{Proposition}
\newtheorem{cor}[thm]{Corollary}
\newtheorem{lem}[thm]{Lemma}
\theoremstyle{definition}

\newtheorem{rem}[thm]{Remark}
\newtheorem*{rem*}{Remark}

\newtheorem*{conclusion*}{Conclusion}

\renewcommand{\phi}{\varphi}

\newcommand{\mathdcl}[1]{{\ifstrequal{#1}{l}{l}{\textup{#1}}}}
\newcommand{\loc}{\mathdcl{loc}}
\newcommand{\cpt}{\mathdcl{c}}
\newcommand{\bdd}{\mathdcl{b}}
\newcommand{\low}{\mathdcl{low}}
\newcommand{\high}{\mathdcl{high}}

\newcommand{\RR}{\mathbb{R}}
\newcommand{\NN}{\mathbb{N}}
\newcommand{\eps}{\varepsilon}
\newcommand{\one}{\mathds{1}}
\newcommand{\setcolon}{\colon}

\newcommand{\clos}[1]{\overline{#1}}
\newcommand{\restrict}[2]{\ensuremath{#1\raisebox{-0.4ex}{$|$}\strut_{#2}}}

\NewDocumentCommand{\dx}{O{x}}{\,\mathrm{d}#1}

\DeclarePairedDelimiter\norm{\lVert}{\rVert}
\DeclarePairedDelimiter\abs{\lvert}{\rvert}

\DeclareMathOperator{\capacity}{cap}
\DeclareMathOperator{\dist}{dist}

\newcommand{\emphdef}[1]{\textbf{\boldmath #1\unboldmath}}

\newlist{romanenum}{enumerate}{1}
\setlist[romanenum]{label=\textup{(\roman*)},ref=\textup{(\roman*)},itemsep=0em,topsep=1ex}

\newlist{alenum}{enumerate}{1}
\setlist[alenum]{label=\textup{(\alph*)},ref=\textup{(\alph*)},itemsep=0em,topsep=1ex}

\newlist{parenum}{enumerate}{1}
\setlist[parenum]{label=\textup{\arabic*.},ref=\textup{\arabic*.},wide,noitemsep}

\newcommand{\niv}{\mathbin{\vrule height 1.6ex depth 0pt width
0.13ex\vrule height 0.13ex depth 0pt width 1.3ex}}

\ExplSyntaxOn
\NewDocumentCommand{\p}{sor()}
{
  \IfBooleanTF{#1}
  {
   \mleft( #3 \mright)
  }{
    \IfNoValueTF{#2}
    {
      \delimitershortfall=10pt\delimiterfactor=700
      \mleft( #3 \mright)
    }{
       \group_begin:
       \mathopen{#2(} #3 \mathclose{#2)}
       \group_end:
    }
  }
}
\ExplSyntaxOff

\title[Variational solutions of the Dirichlet problem]{Variational solutions of the Dirichlet problem, Lebesgue's cusp and non-local properties}

\author[W. Arendt]{Wolfgang Arendt}
\address{Wolfgang Arendt\\Institute of Applied Analysis\\Ulm University\\89069 Ulm\\Germany}
\email{wolfgang.arendt@uni-ulm.de}

\author[D. Daners]{Daniel Daners}
\address{Daniel Daners\\School of Mathematics and Statistics\\The University of Sydney\\NSW 2006\\Australia}
\email{daniel.daners@sydney.edu.au}

\author[M. Sauter]{Manfred Sauter}
\address{Manfred Sauter\\Faculty of Mathematics and Economics\\Ulm University\\89069 Ulm\\Germany}
\email{manfred.sauter@uni-ulm.de}

\dedicatory{Dedicated to Robert Denk on the occasion of his 60\textsuperscript{th} birthday}

\date{11 December 2025}


\begin{document}
\begin{abstract}
A recent result from~\cite{AtES24} allows one to define variational solutions of the Dirichlet problem for general continuous boundary data.
We establish basic properties of this notion of solution and show that it coincides with the Perron solution.
Variational solutions can elegantly be characterised in terms of the given boundary function when the variational solution has finite energy.
However, it is impossible to decide in terms of the regularity of the given boundary function when a classical solution exists.
We demonstrate this by analysing Lebesgue's cusp, and more precisely Lebesgue's domain which is associated with the potential of a thin rod with mass density going to zero at one end.
We also show that the non-continuity of the Perron solution at a singular point is a generic and non-local property.
\end{abstract}

\maketitle

\section{Introduction}

Throughout we suppose that $\Omega$ is a connected bounded open subset of $\RR^d$, $d\ge 2$, with boundary $\partial\Omega$.
Given a function $\phi\in C(\partial\Omega)$, the \emphdef{Dirichlet problem} asks to
\begin{equation}
  \tag*{$D(\varphi)$}
  \label{eq:D}
  \left\{
    \begin{aligned}
      &\text{find $u\in C(\clos{\Omega})$ which is harmonic on $\Omega$} \\
      &\text{such that $\restrict{u}{\partial\Omega}=\phi$.}
    \end{aligned}
  \right.
\end{equation}
We call such a function $u$ a \emphdef{classical solution} of~\ref{eq:D}.
However, a classical solution does not always exist.

In this paper we consider a generalised solution of the Dirichlet problem which is described as follows.
\begin{thm}\label{thm:1.1}
Given $\phi\in C(\partial\Omega)$, there exist a unique harmonic function $u_\phi$ on $\Omega$ and an extension $\Phi\in C(\clos{\Omega})$ of $\phi$ such that $\Phi-u_\phi\in H^1_0(\Omega)$.
\end{thm}
We call $u_\phi$ the \emphdef{variational solution} of~\ref{eq:D}.
It attains the given boundary value $\phi$ in the Sobolev sense.
The existence of such a variational solution is difficult to prove and was an open problem for general $\phi\in C(\partial\Omega)$ until recently.

In the first part of this paper we start with a simple proof of Theorem~\ref{thm:1.1}, making use of a nontrivial extension result that we shall take as given.
We then establish diverse properties of the variational solution defined as above.
We show that it is well-defined; i.e.\ $u_\phi$ does not depend on the extension $\Phi$ of $\phi$ and it coincides with the classical solution of~\ref{eq:D} whenever a classical solution exists.
Moreover, we show that the weak maximum principle is valid for variational solutions.
In Section~\ref{sec:3} we will see that $u_\phi$ coincides with the Perron solution.
This notion was defined by Perron in 1923~\cite{Per1923} and has been a subject of intense research.
One subject is the way in which $u_\phi(x)$ attains the value $\phi(z)$ as $x\to z$.
So Theorem~\ref{thm:1.1} can be interpreted as a new description of the behaviour of the Perron solution at the boundary.

In the second part of the paper, from Section~\ref{sec:4} on, we investigate the behaviour of $u_\phi(x)$ at a fixed singular point $z_0\in\partial\Omega$, i.e.\ a point where $\lim_{x\to z_0} u_\phi(x)=\phi(z_0)$ is violated for some $\phi\in C(\partial\Omega)$.

In a single page paper Lebesgue~\cite{Leb1913} considered a potential in $\RR^3$ of a thin rod modelled by a simple mass distribution concentrated along a line segment.
He chooses two suitable level sets of that potential and considers the bounded, open, simply connected region enclosed by these level sets.
We call this region \emphdef{Lebesgue's domain}.

Lebesgue's domain has a unique singular point $z_0$ which is the tip of an inward pointing cusp, the famous \emphdef{Lebesgue cusp}.
Taking the intersection of Lebesgue's domain with a small ball, one obtains a simply connected domain with connected boundary and a singular point.
This is of great interest since in $\RR^2$ each simply connected domain is \emphdef{Dirichlet regular}; i.e.\ $u_\phi$ is a classical solution for all $\phi\in C(\partial\Omega)$.
Our point is to consider the original Lebesgue domain, whose boundary has two connected components.
It demonstrates two striking features of the generalised solution $u_\phi$.

\begin{enumerate}[1.]
\item It is impossible to deduce from the regularity properties of $\phi$ whether or not $u_\phi$ is a classical solution.
The boundary of Lebesgue's domain is the disjoint union of two closed sets $\Gamma_2$ and $\Gamma_{\frac{1}{2}}$ (which are level sets of the considered potential).
If $\phi$ is constant on each of these two closed sets with two different constants, $u_\phi$ is not a classical solution.
To prove this, the variational definition of $u_\phi$ will be most useful.

\item The irregular behaviour of $u_\phi(x)$ as $x\to z_0$ is a non-local and unstable phenomenon.
For $\phi_0\equiv 0$ one has $u_{\phi_0}\equiv 0$, but if we perturb $\phi_0$ by a small compactly supported positive smooth bump in the neighbourhood of some point far away from $z_0$, then $u_\phi$ is no longer a classical solution.
This can be seen easily at Lebesgue's domain, but we prove it for a singular point of an arbitrary domain in Section~\ref{sec:5}.
\end{enumerate}

As explained above, there seems to be no reasonable condition on the boundary function $\phi$ to ensure that $u_\phi$ is a classical solution of~\ref{eq:D} -- even though there are always many classical solutions.
The situation is different if we instead consider the question of when $u_\phi$ has finite energy $\int_\Omega\abs{\nabla u_\phi}^2$.
This is the case if and only if $\phi$ has an extension $\Phi\in C(\clos{\Omega})\cap H^1(\Omega)$, and then $u_\phi$ may be determined by the Dirichlet principle as the unique minimiser of the Dirichlet energy.
As an easy consequence, if $\phi$ is Lipschitz continuous, then $u_\phi$ has finite energy.
These results are proved in Section~\ref{sec:2} along with other properties of the variational solution.

\section{Variational solutions}\label{sec:2}

By $H^1(\Omega) := \{u\in L^2(\Omega) : \partial_j u\in L^2(\Omega)\text{ for }j=1,\dots,d\}$ we denote the first Sobolev space which is a Hilbert space for the norm
\[
    \norm{u}_{H^1(\Omega)} := \p(\norm{u}_{L^2(\Omega)}^2 + \sum_{j=1}^d\norm{\partial_j u}_{L^2(\Omega)}^2)^{\frac{1}{2}}.
\]
Let $H^1_0(\Omega)$ be the closure of $C^\infty_\cpt(\Omega)$ in $H^1(\Omega)$.
On $H^1_0(\Omega)$ we consider the scalar product
\[
    [v,w] := \int_\Omega\nabla v\nabla w\dx\qquad (v,w\in H^1_0(\Omega)),
\]
which defines an equivalent norm to that induced by $H^1(\Omega)$.

In this section we develop variational solutions of the Dirichlet problem on the basis of the following result~\cite[Corollary~4.3]{AtES24}.
\begin{thm}[Extension result]\label{thm:2.1}
Let $\phi\in C(\partial\Omega)$. Then there exist an extension $\Phi\in C(\clos{\Omega})$ of $\phi$ and $M \ge 0$ such that
\begin{equation}\label{eq:2.1}
    \abs*{\int_\Omega\Phi\Delta w\dx} \le M\norm{w}_{H^1(\Omega)}
\end{equation}
for all $w\in C^\infty_\cpt(\Omega)$. This means that $\Delta\Phi\in H^{-1}(\Omega)$, the dual space of $H^1_0(\Omega)$.
\end{thm}
Note that for $u\in H^1(\Omega)$ one has
\[
    \abs*{\int_\Omega u\Delta w\dx} = \abs*{-\int_\Omega\nabla u\nabla w\dx}\le\norm{u}_{H^1(\Omega)}\norm{w}_{H^1(\Omega)}
\]
for all $w\in C^\infty_\cpt(\Omega)$ and therefore $\Delta u\in H^{-1}(\Omega)$.

Let $\phi\in C(\partial\Omega)$ and $\Phi\in C(\clos{\Omega})$ be an extension of $\phi$ satisfying~\eqref{eq:2.1}.
By the Riesz--Fréchet representation theorem there exists a unique $v\in H^1_0(\Omega)$ such that
\[
    \int_\Omega \Phi\Delta w\dx = -\int_\Omega\nabla v\nabla w\dx = \int_\Omega v\Delta w\dx
\]
for all $w\in C^\infty_\cpt(\Omega)$. This means that $u:=\Phi-v$ satisfies $\Delta u=0$ weakly. Hence $u\in C^\infty(\Omega)$ by Weyl's theorem and $\Delta u=0$ in the classical sense;
that is, $u$ is a harmonic function.
We call $u$ the \emphdef{variational solution} of~\ref{eq:D} and set $u_\phi := u$.
To do so, we have to show that the function does not depend on the extension $\Phi$ of $\phi$. This follows from the next lemma.

\begin{lem}\label{lem:2.2}
Let $v\in C_0(\Omega)$ be such that $\Delta v\in H^{-1}(\Omega)$. Then $v\in H^1_0(\Omega)$.
\end{lem}
In fact, if $\Phi_1,\Phi_2\in C(\clos{\Omega})$ are such that $\restrict{\Phi_1}{\partial\Omega}=\restrict{\Phi_2}{\partial\Omega}$ and $\Delta\Phi_1\in H^{-1}(\Omega)$, $\Delta\Phi_2\in H^{-1}(\Omega)$, then $\Phi_1-\Phi_2\in C_0(\Omega)$ and $\Delta(\Phi_1-\Phi_2)\in H^{-1}(\Omega)$.
Thus $\Phi_1-\Phi_2\in H^1_0(\Omega)$ by Lemma~\ref{lem:2.2}.
Now let $v_1,v_2\in H^1_0(\Omega)$ be such that $\Delta v_1 = \Delta\Phi_1$, $\Delta v_2=\Delta\Phi_2$.
Then $v := \Phi_1-v_1 - (\Phi_2-v_2)\in H^1_0(\Omega)$ and $\Delta v=0$. Thus $[v,v]=\langle-\Delta v,v\rangle = 0$ and so $v=0$.

\begin{proof}[Proof of Lemma~\ref{lem:2.2}]
Let $v\in C_0(\Omega)$ be such that $F := \Delta v\in H^{-1}(\Omega)$.
By the Riesz--Fréchet representation theorem there exists $v_0\in H^1_0(\Omega)$ such that $[v_0,w]=-\int_\Omega v\Delta w\dx$ for all $w\in C^\infty_\cpt(\Omega)$.
Thus
\begin{displaymath}
  \int_\Omega v_0\Delta w\dx = -\int_\Omega\nabla v_0\nabla w\dx = \int_\Omega v\Delta w\dx,
\end{displaymath}
i.e.\ $\Delta(v_0-v)=0$ in the sense of distributions.
Thus $v_0-v\in C^\infty(\Omega)$ by Weyl's theorem.
We conclude that $v\in H^1_\loc(\Omega)$.

Let $\delta>0$. Then $(v-\delta)^+\in H^1_\loc(\Omega)\cap C_\cpt(\Omega)\subset H^1_\cpt(\Omega)$ and $\partial_j(v-\delta)^+=\one_{[v>\delta]}\partial_j v$ (see e.g.~\cite[Section~6.4]{AU23}).
Thus
\[
   \langle-F, (v-\delta)^+\rangle = \int_\Omega\nabla v\nabla (v-\delta)^+\dx = \int_\Omega\abs{\nabla (v-\delta)^+}^2\dx.
\]
It follows that
\[
    \p(\int_\Omega\abs{\nabla (v-\delta)^+}^2\dx)^{\frac{1}{2}}\le\norm{F}_{H^{-1}(\Omega)}.
\]
Since $H^1_0(\Omega)$ is reflexive, we find $v_1\in H^1_0(\Omega)$ and a sequence $\delta_n\downarrow 0$ such that $(v-\delta_n)^+\to v_1$ weakly in $H^1_0(\Omega)$.
Since $(v-\delta_n)^+\to v^+$ in $L^2(\Omega)$, it follows that $v^+=v_1\in H^1_0(\Omega)$. Similarly $v^-\in H^1_0(\Omega)$ and so $v=v^+-v^-\in H^1_0(\Omega)$.
\end{proof}

Next we prove consistency with classical solutions.
\begin{prop}\label{prop:2.3}
Let $\phi\in C(\partial\Omega)$ be such that~\ref{eq:D} has a classical solution $u$. 
Then $u$ is equal to the variational solution $u_\phi$.
\end{prop}
\begin{proof}
Let $\Phi\in C(\clos{\Omega})$ be such that $\restrict{\Phi}{\partial\Omega}=\phi$, $\Delta\Phi\in H^{-1}(\Omega)$.
Then $\Phi-u\in C_0(\Omega)$ and $\Delta(\Phi-u)=\Delta\Phi\in H^{-1}(\Omega)$.
By Lemma~\ref{lem:2.2}, $\Phi-u\in H^1_0(\Omega)$.
Thus $u_\phi=\Phi-(\Phi-u)=u$ by the definition of $u_\phi$.
\end{proof}

We next prove the maximum principle for the variational solution.
\begin{thm}[Maximum principle]\label{thm:2.4}
Let $\phi\in C(\partial\Omega)$ and let $u_\phi$ be the variational solution of~\ref{eq:D}. Then
\begin{equation}\label{eq:2.2}
    \min_{z\in \partial\Omega}\phi(z)\le u_\phi(x)\le\max_{z\in \partial\Omega}\phi(z)
\end{equation}
for all $x\in\Omega$.
\end{thm}
While it is clear that $u_\phi\in C^\infty(\Omega)$ since $u_\phi$ is harmonic,
so far we did not know that $u_\phi$ is bounded. This now follows from~\eqref{eq:2.2}.
In fact,
\begin{equation}
    \norm{u_\phi}_{C_\bdd(\Omega)}\le \norm{\phi}_{C(\partial\Omega)}
\end{equation}
for all $\phi\in C(\partial\Omega)$.
For the proof of Theorem~\ref{thm:2.4} we need some preparation.

We recall Kato's inequality. Let $u\in L^1_\loc(\Omega)$ be such that $\Delta u\in L^1_\loc(\Omega)$. Then
\begin{equation}\label{eq:2.4}
    \one_{[u>0]}\Delta u\le\Delta u^+
\end{equation}
in the weak sense; i.e.
\[
    \int_\Omega \one_{[u>0]}(\Delta u)w\dx \le \int_\Omega u^+\Delta w\dx
\]
for all $0\le w\in C^\infty_\cpt(\Omega)$.
We refer to~\cite[Proposition~5]{AB92} for the proof.
Kato's inequality implies the following form of the maximum principle.

\begin{prop}\label{prop:2.5}
Let $u\in L^1_\loc(\Omega)$ be such that $\Delta u\in L^1_\loc(\Omega)$ and $-\Delta u\le 0$.
Let $c\in\RR$ be such that $(u-c)^+\in H^1_0(\Omega)$. Then $u(x)\le c$ for almost all $x\in\Omega$.
\end{prop}
Note that $(u-c)^+\in H^1_0(\Omega)$ is a weak formulation of $u\le c$ at the boundary.
\begin{proof}[Proof of Proposition~\ref{prop:2.5} {(cf.~\cite[Theorem~3.3]{AtE19:kato})}]
We may assume that $c=0$.
Since $-\Delta u\le 0$, one has by Kato's inequality~\eqref{eq:2.4}
\[
    0\le \one_{[u>0]}\Delta u\le \Delta u^+.
\]
This implies that $\int_\Omega u^+\Delta w\dx\ge 0$ for all $0\le w\in C^\infty_\cpt(\Omega)$.
Since $u^+\in H^1_0(\Omega)$, this implies that $\int_\Omega\nabla u^+\nabla w\dx\le 0$ for all $0\le w\in H^1_0(\Omega)$.
In particular, $\int_\Omega\abs{\nabla u^+}^2\dx\le 0$. Hence $u^+=0$.
\end{proof}

\begin{rem}
In Proposition~\ref{prop:2.5} the hypothesis that $\Delta u\in L^1_\loc(\Omega)$ is not needed.
In fact, let $u\in L^1_\loc(\Omega)$ be such that $-\Delta u\le 0$. Then $\Delta u\in\mathcal{M}(\Omega)$, the space of all Radon measures on $\Omega$; i.e.\ the dual space of $C_\cpt(\Omega)$, see~\cite[Chapter~I, Theorem~V]{Sch1966}.
Now Kato's inequality in the version of Brezis and Ponce~\cite[Theorem~1.1]{BP04} implies that also $\Delta u^+$ is a positive Radon measure (since $(\Delta u^+)_{\mathdcl{d}}\ge 0$ and $(\Delta u^+)_{\mathdcl{c}}\ge 0$, where we use the notation of~\cite{BP04}). If $u^+\in H^1_0(\Omega)$, then the argument in the previous proof shows that $u^+=0$.
\end{rem}

\begin{proof}[Proof of Theorem~\ref{thm:2.4}]
It is clear that $\phi\mapsto u_\phi$ is linear. Moreover, $u_{\one_{\partial\Omega}}=\one_\Omega$ by Proposition~\ref{prop:2.3}.

Let $\phi\in C(\partial\Omega)$, $\phi\le c$. We show that $u_\phi\le c$. This implies the claim.
We may assume that $c=0$ (replacing $\phi$ by $\phi-c$ if necessary since then $u_{\phi-c}=u_\phi -c\one_\Omega$).
Let $\Phi\in C(\clos{\Omega})$ be such that $\Delta\Phi\in H^{-1}(\Omega)$ and $\restrict{\Phi}{\partial\Omega}=\phi\le 0$.
Then $\Phi^+\in C_0(\Omega)$.
Let $v\in H^1_0(\Omega)$ be such that $\Delta v=\Delta\Phi$.
Then $u_\phi=\Phi-v$. Since $u_\phi\in C^\infty(\Omega)$, it follows that $\Phi\in H^1_\loc(\Omega)$.
Let $\eps>0$. We show that $(\Phi-\eps-v)^+\in H^1_0(\Omega)$.
Let $v_n\in C^\infty_\cpt(\Omega)$ be such that $v_n\to v$ in $H^1(\Omega)$.
Then $(\Phi-\eps-v_n)^+\in H^1_\cpt(\Omega)\subset H^1_0(\Omega)$.
Moreover,
\begin{align*}
    0 &= \int_\Omega \Delta u_\phi (\Phi -\eps - v_n)^+\dx \\
        &= - \int_\Omega\nabla u_\phi\nabla (\Phi-\eps-v_n)^+\dx \\
        &= - \int_\Omega\nabla(\Phi -v)\nabla (\Phi-\eps-v_n)^+\dx \\
        &= - \int_\Omega\nabla(\Phi - \eps-v_n)\nabla (\Phi-\eps-v_n)^+\dx + \int_\Omega\nabla(v-v_n)\nabla (\Phi-\eps-v_n)^+\dx\\
        &= -\int_\Omega\abs{\nabla(\Phi-\eps-v_n)^+}^2\dx + \int_\Omega\nabla(v-v_n)\nabla(\Phi-\eps-v_n)^+\dx.
\end{align*}
Hence
\begin{align*}
    \int_\Omega\abs{\nabla(\Phi-\eps-v_n)^+}^2\dx &= \int_\Omega\nabla(v-v_n)\nabla(\Phi-\eps-v_n)^+\dx \\
    &\le\norm{\nabla(v-v_n)}_{L^2}\norm{\nabla(\Phi-\eps-v_n)^+}_{L^2}.
\end{align*}
Therefore
\[
    \norm{\nabla(\Phi-\eps-v_n)^+}_{L^2}\le\norm{\nabla(v-v_n)}_{L^2}\to 0
\]
as $n\to\infty$.
Thus $((\Phi-\eps-v_n)^+)_{n\in\NN}$ is bounded in $H^1_0(\Omega)$.
There exist $\tilde{v}\in H^1_0(\Omega)$ and a subsequence such that $(\Phi-\eps-v_{n_k})^+\to\tilde{v}$ weakly in $H^1_0(\Omega)$ as $k\to\infty$.
Since $(\Phi-\eps-v_{n_k})^+\to(\Phi-\eps-v)^+$ in $L^2(\Omega)$ as $k\to\infty$, it follows that $(\Phi-\eps-v)^+=\tilde{v}\in H^1_0(\Omega)$, which proves the claim.

Since $\Delta(\Phi-\eps-v)=\Delta(\Phi-v)=\Delta u_\phi=0$, and $(\Phi-\eps-v)^+\in H^1_0(\Omega)$ for all $\eps>0$, it follows from Proposition~\ref{prop:2.5} that $u_\phi\le 0$.
We have shown the inequality on the right of~\eqref{eq:2.2}. The inequality on the left follows from the one on the right by replacing $\phi$ by $-\phi$.
\end{proof}

Next we talk about the Dirichlet principle.
We start by mentioning that the variational solution $u_\phi$ is obtained by a minimising process.
In fact, let $\phi\in C(\partial\Omega)$ and $\Phi\in C(\clos{\Omega})$ be an extension of $\phi$ satisfying the estimate~\eqref{eq:2.1}, i.e.\ $\Delta\Phi\in H^{-1}(\Omega)$.
The solution $v\in H^1_0(\Omega)$ of $\Delta v=\Delta\Phi$ is the unique function $v\in H^1_0(\Omega)$ such that
\begin{equation}\label{eq:2.5}
    \frac{1}{2}\int_\Omega\abs{\nabla v}^2\dx + \langle\Delta\Phi,v\rangle = \min\Bigl\{\frac{1}{2}\int_\Omega\abs{\nabla w}^2\dx + \langle\Delta\Phi,w\rangle \setcolon w\in H^1_0(\Omega)\Bigr\}
\end{equation}
(see e.g.~\cite[Theorem~4.24]{AU23}).
Now we can characterise when $\int_\Omega\abs{\nabla u_\phi}^2\dx<\infty$; i.e.\ when $u_\phi$ has finite energy.
Since $u_\phi\in C^\infty(\Omega)\cap C_\bdd(\Omega)$, this is equivalent to saying that $u_\phi\in H^1(\Omega)$.

\begin{thm}\label{thm:2.7}
Let $\phi\in C(\partial\Omega)$. The following statements are equivalent.
\begin{romanenum}
\item\label{en:t2.7-i}
The variational solution $u_\phi$ of~\ref{eq:D} is an element of $H^1(\Omega)$.
\item\label{en:t2.7-ii}
There exists $\Phi\in C(\clos{\Omega})\cap H^1(\Omega)$ with $\restrict{\Phi}{\partial\Omega}=\phi$.
\end{romanenum}
In that case
\[
    \int_\Omega\abs{\nabla u_\phi}^2\dx = \min\Bigl\{\int_\Omega\abs{\nabla w}^2\dx \setcolon w\in H^1(\Omega),\ \Phi-w\in H^1_0(\Omega)\Bigr\}.
\]
\end{thm}
\begin{proof}
\begin{parenum}
\item[\ref{en:t2.7-ii}$\Rightarrow$\ref{en:t2.7-i}.]
Since $\Phi\in H^1(\Omega)$, $u_\phi=\Phi-v\in H^1(\Omega)$, where $v\in H^1_0(\Omega)$ is such that $\Delta v=\Delta\Phi$.
Moreover, by~\eqref{eq:2.5} and considering $W:=\{w\in H^1(\Omega)\setcolon \Phi-w\in H^1_0(\Omega)\}$,
\begin{align*}
    \frac{1}{2}[u_\phi,u_\phi] -\frac{1}{2}[\Phi,\Phi] &= \frac{1}{2}[v,v]-[\Phi,v] \\
        &= \min\Bigl\{\frac{1}{2}[\Phi-w,\Phi-w] - [\Phi,\Phi-w] \setcolon w\in W\Bigr\} \\
        &= \min\Bigl\{\frac{1}{2}[w,w] - \frac{1}{2}[\Phi,\Phi] \setcolon w\in W\Bigr\} \\
        &= \frac{1}{2} \min\Bigl\{[w,w] \setcolon w\in W\Bigr\}-\frac{1}{2}[\Phi,\Phi].
\end{align*}
Thus $[u_\phi,u_\phi] = \min\{[w,w]\setcolon w\in W\}$ and this property characterises $u_\phi$.
\item[\ref{en:t2.7-i}$\Rightarrow$\ref{en:t2.7-ii}.]
Let $\Phi\in C(\clos{\Omega})$ be such that $\restrict{\Phi}{\partial\Omega}=\phi$ and $\Delta\Phi\in H^{-1}(\Omega)$.
Let $v\in H^1_0(\Omega)$ be such that $\Delta v=\Delta\Phi$.
Then $u_\phi=\Phi-v$.
Since $u_\phi\in H^1(\Omega)$ by hypothesis, it follows that $\Phi\in H^1(\Omega)$.\qedhere
\end{parenum}
\end{proof}

\begin{cor}
Let $\phi\in C(\partial\Omega)$ be Lipschitz continuous. Then $u_\phi\in H^1(\Omega)$.
\end{cor}
\begin{proof}
By~\cite[Theorem~1]{McS34} or~\cite[Theorem~3.1]{EG2015} there exists a Lipschitz continuous extension $\Phi\in C(\RR^d)$ of $\phi$.
It follows from~\cite[Theorem~4.5]{EG2015} that $\restrict{\Phi}{\Omega}\in W^{1,\infty}(\Omega)\subset H^1(\Omega)$.
\end{proof}

Prym~\cite{Pry71} in 1871 and independently Hadamard~\cite{Had06} in 1906 showed that for $\mathds{D}=\{x\in\RR^2 \setcolon \abs{x}<1\}$ there exists $\phi\in C(\partial\mathds{D})$ such that $\int_{\mathds{D}}\abs{\nabla u_\phi}^2=\infty$ (see~\cite[Section~6.9]{AU23} and~\cite[Section~12.3]{MS98}).
Thus, this $\phi$ does not have any extension $\Phi\in C(\clos{\mathds{D}})\cap H^1(\mathds{D})$.

If $\Omega$ is more regular, the following holds.
\begin{thm}\label{thm:2.9}
Assume that $\Omega$ has Lipschitz boundary and let $\phi\in C(\partial\Omega)$. The following statements are equivalent.
\begin{romanenum}
\item $u_\phi\in H^1(\Omega)$
\item $\phi\in H^{1/2}(\partial\Omega)$
\end{romanenum}
\end{thm}
Here, since $\Omega$ has Lipschitz boundary, $H^{1/2}(\partial\Omega)=\{\operatorname{tr}\Phi \setcolon \Phi\in H^1(\Omega)\}$, where $\operatorname{tr}\colon H^1(\Omega)\to L^2(\partial\Omega,\mathcal{H}^{d-1})$ is the trace operator and $\mathcal{H}^{d-1}$ denotes the $(d-1)$-dimensional Hausdorff measure.
If $\phi\in C(\partial\Omega)\cap H^{1/2}(\partial\Omega)$, there even exists a $\Phi\in C(\clos{\Omega})\cap H^1(\Omega)$ such that $\restrict{\Phi}{\partial\Omega}=\phi$.
We refer to~\cite[Theorem~1.2]{AtE19:dp} for this and a proof of Theorem~\ref{thm:2.9}.
A more general result on very rough domains is~\cite[Theorem~5.8]{AtES24}.

While we focus on general irregular bounded domains with continuous boundary data here,
one can also consider elliptic problems on more regular domains, but with more irregular boundary data. For information on elliptic problems with rough boundary data we refer to the recent papers~\cite{ADM21} by Anop, Denk and Murach and~\cite{DPRS23} by Denk, Plo{\ss}, Rau and Seiler.

\section{Quasi-everywhere convergence and the Perron solution}
\label{sec:3}
We now investigate the pointwise behaviour of $u_\phi(x)$ as $x\to z$ where $z\in \partial\Omega$.

For $A\subset\RR^d$, we define the (Sobolev) \emphdef{capacity} of $A$ as
\[
    \capacity(A) = \inf\{ \norm{v}_{H^1(\RR^d)}^2 \setcolon \text{$v\in H^1(\RR^d)$,
        $v\ge 1$ in a neighbourhood of $A$}\}.
\]
If $\capacity(A)=0$, then clearly $\abs{A}=0$ (where $\abs{A}$ denotes the $d$-dimensional Lebesgue measure).
If $d\ge 2$, one has $\capacity(\{x\})=0$ for all $x\in\RR^d$. In dimension $d\ge 3$, even lines in $\RR^d$ have capacity zero.
Still, in $\RR^2$ any proper line segment has strictly positive capacity.
We say that a property $P(x)$ holds for \emphdef{quasi-every} $x\in A$ for a subset $A\subset\RR^d$, if the set $\{x\in A\setcolon\text{$P(x)$ fails}\}$ has capacity $0$.
With a similar but slightly more complicated proof than the one of Lemma~\ref{lem:2.2} we will show the following.

\begin{lem}\label{lem:3.1}
Let $v\in C_\bdd(\Omega)$ be such that $\Delta v\in H^{-1}(\Omega)$ and
\[
    \lim_{x\to z} v(x)=0
\]
for quasi-every $z\in \partial\Omega$. Then $v\in H^1_0(\Omega)$.
\end{lem}
\begin{proof}[Proof ({cf.~\cite[Lemma~3.11]{AtES24}})]
Let $F=\Delta v\in H^{-1}(\Omega)$ and let $P\subset \partial\Omega$ be such that $\capacity(P)=0$ and $\lim_{x\to z} v(x)=0$ for all $z\in \partial\Omega\setminus P$.
As in the proof of Lemma~\ref{lem:2.2} one obtains $v\in H^1_\loc(\Omega)$.
There exist $\chi_n\in H^1(\RR^d)$ such that $0\le\chi_n\le 1$, $\norm{\chi_n}_{H^1(\RR^d)}\le\frac{1}{n}$ and $\chi_n=1$ in a neighbourhood of $P$.
For $\delta>0$ and $n\in\NN$ define
\[
    v_{n,\delta} := (1-\chi_n)^2(v-\delta)^+\in H^1_\cpt(\Omega).
\]
Then
\[
    \nabla v_{n,\delta} = -2(1-\chi_n)(v-\delta)^+\nabla\chi_n + (1-\chi_n)^2\one_{[v>\delta]}\nabla v.
\]
Since $\Delta v=F$ and $v\in H^1_\loc(\Omega)$, one has
\[
    \abs*{\int_\Omega\nabla v\nabla w\dx}=\abs{F(w)}\le M_1\p(\int_\Omega\abs{\nabla w}^2\dx)^{\frac{1}{2}}
\]
for all $w\in H^1_\cpt(\Omega)$ and some $M_1\ge 0$.
It follows that
\begin{align*}
    M_1\norm{\nabla v_{n,\delta}}_{L^2} &\ge \int_\Omega\nabla v\nabla v_{n,\delta}\dx \\
        &=A_{n,\delta}^2 - 2\int_\Omega(1-\chi_n)(v-\delta)^+\nabla v\nabla\chi_n\dx,
\end{align*}
where $A_{n,\delta}=\norm{(1-\chi_n)\one_{[v>\delta]}\nabla v}_{L^2}$.

On the other hand, using that $v$ is bounded, we obtain
\begin{align*}
    \norm{\nabla v_{n,\delta}}_{L^2}&\le C_1 + \norm{(1-\chi_n)^2\one_{[v>\delta]}\nabla v}_{L^2} \\
        &\le C_1 + \norm{(1-\chi_n)\one_{[v>\delta]}\nabla v}_{L^2} = C_1 + A_{n,\delta}
\end{align*}
for some constant $C_1\ge0$ independent of $n$ and $\delta>0$.
Thus
\begin{align*}
  A_{n,\delta}^2
  &\le M_1(C_1+A_{n,\delta}) + \abs*{\int_\Omega 2(1-\chi_n)(v-\delta)^+\nabla v\nabla\chi_n\dx}\\
  &\le M_1(C_1+A_{n,\delta}) + C_2 A_{n,\delta}
\end{align*}
since $v$ is bounded.
This implies that $A_{n,\delta}\le C_3$ for all $n\in\NN$ and $\delta>0$.
We may assume that $\chi_n\to 0$ a.e. By Fatou's lemma
\[
    \p(\int_\Omega\abs{\nabla(v-\delta)^+}^2\dx)^{\frac{1}{2}}\le\liminf_{n\to\infty} A_{n,\delta}\le C_3
\]
for all $\delta>0$. As in the proof of Lemma~\ref{lem:2.2} this implies that $v^+\in H^1_0(\Omega)$.
Replacing $v$ by $-v$ we obtain $v^-\in H^1_0(\Omega)$ and so $v=v^+-v^-\in H^1_0(\Omega)$.
\end{proof}

We recall the definition of the Perron solution of~\ref{eq:D}.
Given $\phi\in C(\partial\Omega)$, a function $u\in C(\clos{\Omega})$ is called a \emphdef{subsolution} of~\ref{eq:D} if
\begin{alenum}
\item $\limsup_{x\to z} u(x)\le \phi(z)$ for all $z\in \partial\Omega$, and
\item $-\Delta u\le 0$ weakly; i.e.\ $-\int_\Omega u\Delta w\dx\le 0$ for all $0\le w\in C^\infty_\cpt(\Omega)$.
\end{alenum}
Similarly one defines supersolutions.
Then Perron~\cite{Per1923} considers
\begin{align*}
    \overline{u}(x) &:= \inf\{w(x) \setcolon \text{$w$ is a supersolution of~\ref{eq:D}}\}
\shortintertext{and}
    \underline{u}(x) &:= \sup\{w(x) \setcolon \text{$w$ is a subsolution of~\ref{eq:D}}\}
\end{align*}
for all $x\in\Omega$. He shows that $\underline{u}=\overline{u}\in C^\infty(\Omega)$, $\Delta \underline{u}=0$ and that $\underline{u}=u$ whenever~\ref{eq:D} has a classical solution $u$.
The function $\underline{u}=\overline{u}$ is called the \emphdef{Perron solution} of~\ref{eq:D}.
We refer to Gilbarg and Trudinger~\cite[Sections~2.8 and~2.9]{GT01} for a brief introduction and to Dautray and Lions~\cite[Chapter~2]{DL90:vol1} for more information.

For the Perron solution the following characterisation due to Vasilesco~\cite{Vas35} in terms of the boundary behaviour is known.
\begin{thm}\label{thm:3.2}
Let $\phi\in C(\partial\Omega)$ and let $u$ be a bounded harmonic function on $\Omega$. The following statements are equivalent.
\begin{romanenum}
\item $\lim_{x\to z}u(x)=\phi(z)$ for quasi-every $z\in \partial\Omega$.
\item $u$ is the Perron solution of~\ref{eq:D}.
\end{romanenum}
\end{thm}
Now we deduce the following result from Lemma~\ref{lem:3.1}.

\begin{thm}
Let $\phi\in C(\partial\Omega)$. Then the variational solution $u_\phi$ of~\ref{eq:D} coincides with the Perron solution of~\ref{eq:D}.
\end{thm}
\begin{proof}
Let $\Phi\in C(\clos{\Omega})$ be such that $\restrict{\Phi}{\partial\Omega}=\phi$ and $\Delta\Phi\in H^{-1}(\Omega)$.
Let $u$ be the Perron solution. Then $v:=\Phi-u\in C_\bdd(\Omega)$ and $\lim_{x\to z} v(x)=0$ for quasi-every $z\in \partial\Omega$. Since $\Delta v = \Delta\Phi\in H^{-1}(\Omega)$, it follows from Lemma~\ref{lem:3.1} that $v\in H^1_0(\Omega)$.
Thus $u_\phi=\Phi-v=u$.
\end{proof}

The variational solution was first defined in~\cite[Theorem~1.1]{AD08:var} and it was shown that it coincides with the Perron solution using a completely different characterisation of the Perron solution based on domain approximation from the interior.
However, at that time the extension result, Theorem~\ref{thm:2.1}, was not known. So the variational solution was only defined provided such an extension was available and it was formulated as an open problem in~\cite[Remark~1.6\,(b)]{AD08:var} whether such extensions always exist.

In~\cite{AtES24} a generalised solution is defined differently. For $\phi=\restrict{\Phi}{\partial\Omega}$ with $\Phi\in C(\clos{\Omega})\cap H^1(\Omega)$, let $T\phi := u_\phi$ with our definition of $u_\phi$.
Then $\norm{T\phi}_{C_\bdd(\Omega)}\le \norm{\phi}_{C(\partial\Omega)}$ by the maximum principle (which is much easier to prove here since $\Phi\in H^1(\Omega)$).
So $T$ has a unique contractive extension $\widetilde{T}\colon C(\partial\Omega)\to C_\bdd(\Omega)$, and $\widetilde{T}\phi$ is then investigated for arbitrary $\phi$.
It is shown that $\widetilde{T}\phi$ coincides with the Perron solution for all $\phi\in C(\partial\Omega)$.

The technically demanding proof of the extension result, Theorem~\ref{thm:2.1}, actually works by showing that the boundary behaviour of a function $w\in C_\bdd(\Omega)$ such that $\Delta w\in H^{-1}(\Omega)$ and $\lim_{x\to z}w(x)=\phi(z)$ for quasi-every $z\in \partial\Omega$ can be iteratively improved by small $H^1_0(\Omega)$ perturbations. After a limiting process of such improvements, one arrives at an extension $\Phi\in C(\clos{\Omega})$ with $\Delta\Phi\in H^{-1}(\Omega)$.
In~\cite{AtES24} the solution operator $\widetilde{T}$ is defined for more general elliptic operators and in a setting where no maximum principle is available, and the extension result is then proved in~\cite[Corollary~4.3]{AtES24} using the solution $\widetilde{T}\phi$ as the starting function $w$.
The fact that here we consider the Laplacian instead of more general elliptic operators as in~\cite{AtES24} does not lead to a significantly simplified proof of Theorem~\ref{thm:2.1}.

\section{Lebesgue's cusp and domain}\label{sec:4}

We study the pointwise convergence of $u_\phi(x)$ as $x\to z$, where $z$ is a point on the boundary of $\Omega$.
A point $z\in \partial\Omega$ is called \emphdef{regular} for the Dirichlet problem on $\Omega$ if $\lim_{x\to z}u_\phi(x)=\phi(z)$ for all $\phi\in C(\clos{\Omega})$, and $z$ is \emphdef{singular} if it is not regular.
By Kellogg's theorem~\cite[Theorem~6.3.4]{AH96} the set of all singular points in $\partial\Omega$ has capacity $0$.
In dimension $d=2$ a point $z\in \partial\Omega$ is regular whenever there is a line segment $[z,b] := \{(1-t)z+tb \setcolon 0\le t\le 1\}$ with $b\in\RR^2\setminus\{z\}$ such that $\clos{\Omega}\cap[z,b] = \{z\}$.
This is no longer true if $d=3$ as we will see below.
But if there exists a proper circular finite closed cone $C_z$ with vertex $z$ such that $C_z\cap\clos{\Omega}=\{z\}$, then $z$ is a regular point for $d\ge 3$.

We say that $\Omega$ is \emphdef{Dirichlet regular}, if each $z\in \partial\Omega$ is regular.
This is equivalent to saying that for each $\phi\in C(\partial\Omega)$ there exists a (necessarily unique) classical solution of~\ref{eq:D}, and $u_\phi$ is this classical solution by Proposition~\ref{prop:2.3}.
In dimension $d=2$ every simply connected bounded open set is Dirichlet regular; see~\cite[Theorem~4.2.1]{Ran95} or \cite[Corollary~X.4.18]{Con78}.
But given $x_0\in\Omega$, the set $\Omega\setminus\{x_0\}$ is not regular.
In dimension $d\ge 3$, $\Omega$ is Dirichlet regular if it satisfies the exterior cone condition (as we have mentioned above); and so, in particular, if $\Omega$ has Lipschitz boundary.
But already for $d=3$ it does not suffice that $\Omega$ is simply connected or that $\Omega$ has continuous boundary (see e.g.~\cite[Remark~7.2\,(d)]{AU23} for the definition).
A counterexample is given by Lebesgue's famous cusp.
Actually, Lebesgue defines a domain $\Omega$ in $\RR^3$ which is described in his single page article from 1913~\cite{Leb1913} shown in full in Figure~\ref{fig:lebesgue-paper}.

\begin{figure}[th]
\centering
\fbox{\includegraphics[totalheight=0.87\textheight]{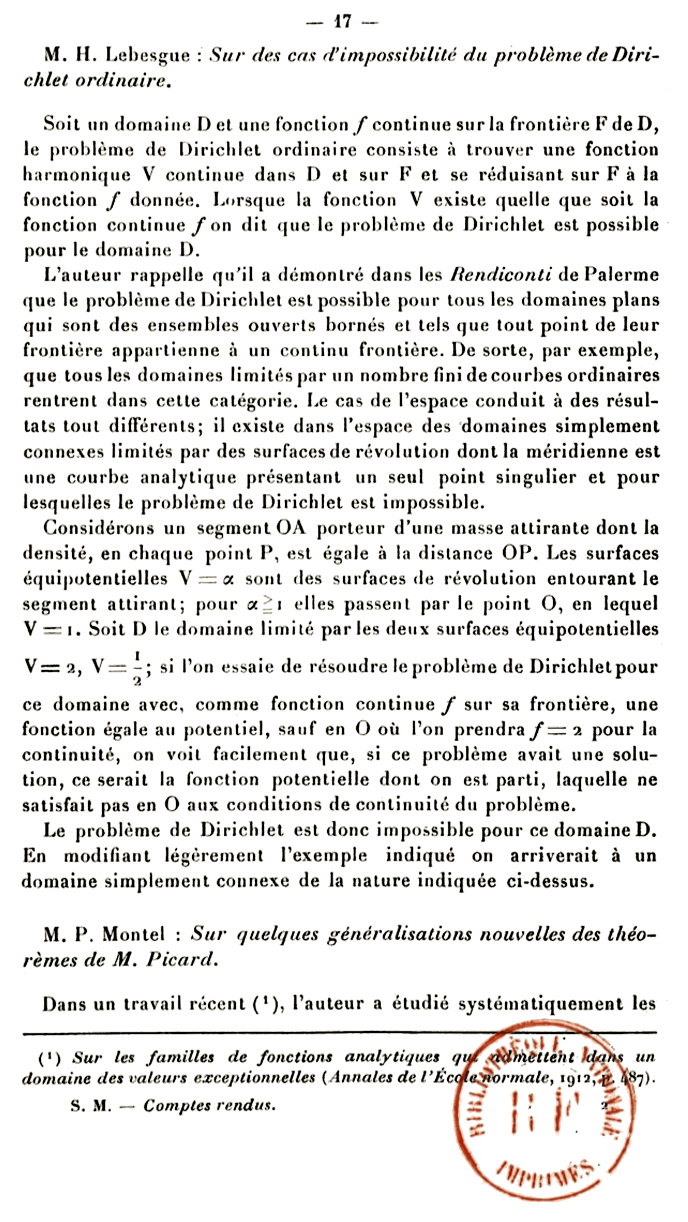}}
\caption{Lebesgue's paper \cite{Leb1913} (Source \href{https://gallica.bnf.fr/ark:/12148/bpt6k9446739}{\nolinkurl{gallica.bnf.fr}} / Bibliothèque nationale de France).}
\label{fig:lebesgue-paper}
\end{figure}
We now describe Lebesgue's domain in more detail.
It is remarkable that it is defined by a physical situation in terms of a mass (or charge) distribution on a segment $S:=\{(0,0,z)\setcolon 0\le z\le L\}$.
Lebesgue considers $L=1$ and the mass density $\rho(z)=z$, $z\in[0,1]$. We will start by considering more general densities. Assume that $\rho\in C([0,L];[0,\infty))$ with $\rho(0)=0$ and $\rho(z)>0$ for all $z\in (0,L]$. This gives rise to the Radon measure $\mu=\tilde{\rho}\cdot\mathcal{H}^1\niv S$ on $\RR^3$, where $\mathcal{H}^1$ denotes the \mbox{$1$-dimensional} Hausdorff measure and $\tilde{\rho}(x,y,z)=\rho(z)$ for $z\in[0,L]$ and $\tilde{\rho}(x,y,z)=0$ otherwise. The mass distribution is radially symmetric with respect to the $z$-axis, so we use cylindrical coordinates to represent the potential $V$. Writing $r:=\sqrt{x^2+y^2}$, it is given by the convolution of the measure $\mu$ with the Newtonian potential, that is,
\begin{equation}
  \label{eq:lebesgue-potential}
  V(r,z)=\int_0^L\frac{\rho(\zeta)}{\sqrt{(\zeta-z)^2+r^2}}\dx[\zeta].
\end{equation}
We assume that
\begin{equation}
  \label{eq:lebesgue-potential-critical}
  V(0,0):= \int_0^L\frac{\rho(\zeta)}{\zeta}\dx[\zeta]<\infty.
\end{equation}
This is clearly the case for Lebesgue's choice $L=1$ and $\rho(z)=z$, resulting in $V(0,0)=1$. Then $V$ is harmonic on the domain
\begin{displaymath}
  D:=\RR^3\setminus\{(0,0,z)\setcolon 0\le z\le L\}.
\end{displaymath}
We start by establishing a result that allows us to analyse the level sets of $V$.
\begin{lem}
  \label{lem:lebesgue-potential-basics}
  Let $V\colon D\to(0,\infty)$ be the potential defined by~\eqref{eq:lebesgue-potential} and suppose that~\eqref{eq:lebesgue-potential-critical} is satisfied. Then the following assertions are true.
  \begin{alenum}
  \item\label{itm:V-monotone} For every $z\in\RR$ the map $V(\cdot\,,z)$ is strictly decreasing on $(0,\infty)$. Moreover,
    \begin{displaymath}
      \lim_{r\to 0+}V(r,z)=
      \begin{dcases}
        V(0,z)&\text{if }z\notin(0,L],\\
        \infty&\text{if }z\in(0,L].
      \end{dcases}
    \end{displaymath}
    and
    \begin{displaymath}
      \lim_{|z|+r\to\infty}V(r,z)=0.
    \end{displaymath}
  \item\label{itm:V-contour-regular} For all $(r,z)\in D$ we have $(\nabla V)(r,z)\neq 0$.
  \item\label{itm:V-continuous-0} If $\alpha\in[0,\pi/2)$ and $D_\alpha:=\{(r,z)\in D\setcolon z\le\tan(\alpha)r\}$, then
    \begin{equation}
      \label{eq:V-sector-domination}
      V(r,z)\le\sec(\alpha)V(0,0)
    \end{equation}
    for all $(r,z)\in D_\alpha$ and  $V\colon D_\alpha\cup\{(0,0)\}\to(0,\infty)$ is continuous.
  \end{alenum}
\end{lem}
\begin{proof}
  \ref{itm:V-monotone} The fact that $V(\cdot\,,z)$ is strictly decreasing for all $z\in\RR$ is obvious from~\eqref{eq:lebesgue-potential}. We also note that for every $z\in(0,L]$
  \begin{displaymath}
    \int_0^z\frac{\rho(\zeta)}{z-\zeta}\dx[\zeta]=\infty.
  \end{displaymath}
  By the monotone convergence theorem we therefore have that
  \begin{displaymath}
    \lim_{r\to 0+}\int_0^L\frac{\rho(\zeta)}{\sqrt{(\zeta-z)^2+r^2}}\dx[\zeta]
    =\int_0^L\frac{\rho(\zeta)}{|\zeta-z|}\dx[\zeta]
    =
    \begin{dcases}
      V(0,z)&\text{if }z\notin(0,L],\\
      \infty&\text{if }z\in(0,L],
    \end{dcases}
  \end{displaymath}
  where the case $z=0$ made use of~\eqref{eq:lebesgue-potential-critical}.

  \ref{itm:V-contour-regular} This follows since $V(\cdot\,,z)$ is strictly decreasing for all $z\in\RR$ and since $V(0,\cdot)$ is strictly monotone as a function of $z\in(-\infty,0]\cup(L,\infty)$.

  \ref{itm:V-continuous-0} If $z\le 0$, then
  \begin{displaymath}
    \frac{\rho(\zeta)}{\sqrt{(\zeta-z)^2+r^2}}
    \le\frac{\rho(\zeta)}{\zeta}
    \le\frac{\rho(\zeta)}{\zeta}\sec(\alpha)
  \end{displaymath}
  for all $r>0$ and $\zeta\in(0,L)$. Assume that $(r,z)\in D_\alpha$ with $z>0$. Then,
  \begin{displaymath}
    \frac{\rho(\zeta)}{\sqrt{(\zeta-z)^2+r^2}}
    =\frac{\rho(\zeta)}{\zeta}\frac{\zeta}{\sqrt{(\zeta-z)^2+r^2}}
    \le\frac{\rho(\zeta)}{\zeta}\frac{\zeta}{\sqrt{(\zeta-z)^2+\dfrac{z^2}{\tan^2(\alpha)}}}
  \end{displaymath}
  for all $\zeta\in(0,L]$. Note that
  \begin{displaymath}
    \frac{\zeta}{\sqrt{(\zeta-z)^2+\dfrac{z^2}{\tan^2(\alpha)}}}
    =\frac{1}{\sqrt{\mleft(1-\dfrac{z}{\zeta}\mright)^2+\dfrac{(z/\zeta)^2}{\tan^2(\alpha)}}}\le\sqrt{1+\tan^2(\alpha)}=\sec(\alpha)
  \end{displaymath}
  if we minimise the polynomial in the square root over $t=z/\zeta>0$. This proves~\eqref{eq:V-sector-domination}. As a consequence of the above estimates, the dominated convergence theorem implies that
  \begin{displaymath}
    \lim_{\substack{(r,z)\to (0,0)\\(r,z)\in D_\alpha}}V(r,z)
    =\lim_{\substack{(r,z)\to (0,0)\\(r,z)\in D_\alpha}}\int_0^L\frac{\rho(\zeta)}{\sqrt{(\zeta-z)^2+r^2}}\dx[\zeta]
    =\int_0^L\frac{\rho(\zeta)}{\zeta}\dx[\zeta]
    =V(0,0),
  \end{displaymath}
  proving the continuity of $V$ at $(0,0)$ as a function of $(r,z)\in D_\alpha$.
\end{proof}
Lebesgue defined his domain as the region between two level surfaces of $V$. We now look at the properties of the level surfaces. For $c\in\RR$ we let
\begin{equation}
  \label{eq:V-level-set}
  \Gamma_c:=\mleft\{(x,y,z)\in D\setcolon V\mleft(\sqrt{x^2+y^2},z\mright)=c\mright\}.
\end{equation}
We note that $\Gamma_c$ can also be considered as the surface of revolution obtained by revolving the contour line
\begin{equation}
  \label{eq:V-contour-lines}
  L_c:=\{(r,z)\in[0,\infty)\times\RR\setcolon (r,z)\in D\text{ and }V(r,z)=c\}\subset\RR^2\end{equation}
about the $z$-axis. We next collect properties of $\Gamma_c$ and $L_c$.
\begin{prop}
  \label{prop:lebesgue-contours}
  Suppose that $V$ is given by~\eqref{eq:lebesgue-potential}, where $\rho\in C([0,L];[0,\infty))$ with $\rho(0)=0$ and $\rho(z)>0$ for all $z\in(0,L]$. Further assume that~\eqref{eq:lebesgue-potential-critical} holds. Let $c\in(0,\infty)$. Then one has the following properties:
  \begin{romanenum}
  \item\label{itm:lebesgue-contours-i} $L_c\neq\emptyset$ and there exists a compact interval $[z_1,z_2]$ with $[0,L]\subset[z_1,z_2)$ and a function $r_c\in C([z_1,z_2];[0,\infty))$ such that $r_c$ is analytic and strictly positive on $(z_1,z_2)$, $r_c(z_1)=r_c(z_2)=0$ and
\begin{equation}\label{eq:contour-function}
    L_c\cup\{(z_1,0)\}=\{(z,r_c(z))\setcolon z\in[z_1,z_2]\}.
\end{equation}
  \item\label{itm:lebesgue-contours-ii} If $c<V(0,0)$, then $z_1<0$, $(z_1,0)\in L_c$ and $\Gamma_c$ is an analytic closed surface in $D$.
  \item\label{itm:lebesgue-contours-iii} If $c>V(0,0)$, then the curve $(z,r_c(z))$ approaches $(0,0)$ tangentially along the $z$-axis as $z\to 0+$, that is, $\Gamma_c$ is an analytic surface in $D$ with a cusp towards $(0,0,0)$ and $\clos{\Gamma_c}=\Gamma_c\cup\{(0,0,0)\}$.
  \end{romanenum}
\end{prop}
\begin{proof}
Due to the monotonicity of $V(\cdot,z)$ for any fixed $z\in\RR$, there exists at most one $r\ge 0$ with $(r,z)\in D$ and $V(r,z)=c$; cf.~Lemma~\ref{lem:lebesgue-potential-basics}~\ref{itm:V-monotone}.
Observe that $V(0,\cdot)$ is strictly increasing in $z$ from $0$ to $V(0,0)$ on $(-\infty,0]$, and strictly decreasing from $\infty$ to $0$ on $(L,\infty)$.
So there exists a unique $z_2\in(L,\infty)$ such that $V(0,z_2)=c$.
In particular $L_c\ne\emptyset$.
Moreover, if $c\le V(0,0)$, there exists a unique $z_1\in (-\infty,0]$ such that $V(0,z_1)=c$.
Otherwise, if $c>V(0,0)$, there is no $(r,z)\in[0,\infty)\times(-\infty,0]$ with $V(r,z)=c$.
We then set $z_1=0$. Obviously in both cases $[0,L]\subset[z_1,z_2)$.

For $z\in(0,L]$ by Lemma~\ref{lem:lebesgue-potential-basics}~\ref{itm:V-monotone},
and otherwise for $z\notin[0,L]$ by the aforementioned monotonicity properties and continuity of $V$ in $(0,0,z)$ for $z\notin [0,L]$,
we deduce that for every $z\in(z_1,z_2)$ there exists a unique $r_c(z):=r>0$ with $V(r,z)=c$, and for every $z\in\RR\setminus[z_1,z_2]$ there is no $r\ge 0$ with $V(r,z)=c$. We set $r_c(z_1)=r_c(z_2)=0$ to obtain a function $r_c\colon[z_1,z_2]\to[0,\infty)$ that satisfies~\eqref{eq:contour-function}.

Next note that $\nabla V\neq 0$ everywhere in $D$. So by the implicit function theorem $\Gamma_c$ is an analytic surface in $D$ and the function $r_c$ is analytic and strictly positive on $(z_1,z_2)$.
Continuity of $r_c$ in $z_2$ follows from analyticity of $\Gamma_c$ at $(0,0,z_2)$, and analogously continuity in $z_1$ if $c<V(0,0)$. This also establishes~(ii).

We claim that $r_c$ is continuous in $z_1$ also for $c\ge V(0,0)$.
In fact, assume for contradiction that $\limsup_{w\to 0+} r_c(w)>0$.
Then there exists a $\beta>0$ and a strictly decreasing sequence $(w_n)_{n\in\NN}$ such that $\lim_{n\to\infty}w_n=0$ and $\frac{1}{\beta}\ge r_c(w_n)\ge\beta$ for all $n\in\NN$.
So after going to a subsequence we may suppose that $(r_c(w_n),w_n)$ converges to $(r_0,0)$ with $r_0\ge\beta$. But $V(r_0,0)<V(0,0)\le c$, which is a contradiction to continuity of $V$ on $D$ as $V(r_c(w_n),w_n)=c$ for all $n\in\NN$.

It remains to prove (iii). So let $c>V(0,0)$.
By Lemma~\ref{lem:lebesgue-potential-basics}~\eqref{eq:V-sector-domination} it is impossible for $\Gamma_c$ to stay in $D_\alpha$ for any $\alpha\in[0,\pi/2)$. Hence $(z,r_c(z))$ must approach $(0,0)$ tangentially along the $z$-axis as $z\to 0+$.
\end{proof}

\begin{cor}\label{cor:sing-dom}
Adopt the assumptions of Proposition~\ref{prop:lebesgue-contours}.
Let $A,B\in\RR$ with $0<A<V(0,0)<B$. Then
\begin{align*}
    \Omega := \{ (x,y,z)\in D: A<V(\sqrt{x^2+y^2},z)<B \}
\end{align*}
defines a bounded simply connected domain in $\RR^3$ whose boundary $\partial\Omega$ consists of the
two connected components $\Gamma_A$ and $\Gamma_B\cup\{(0,0,0)\}$ and
such that the boundary point $(0,0,0)$ is singular for the Dirichlet problem on $\Omega$.
\end{cor}
\begin{proof}
In this proof we write $x$ for points in $\Omega\subset\RR^3$ and $z$ for points on $\partial\Omega$.
It remains to show that $z_0:=(0,0,0)$ is singular for $\Omega$.
Define $\phi\in C(\Gamma)$ by letting $\phi=A$ on $\Gamma_A$ and $\phi=B$ on $\Gamma_B\cup\{(0,0,0)\}$.
Let $u := \restrict{V}{\Omega}$. Then $u$ is harmonic and bounded on $\Omega$.
Moreover, $\lim_{x\to z} u(x)=\phi(z)$ for all $z\in\Gamma_A\cup\Gamma_B$.

It now suffices to observe that the singleton set $\{z_0\}$ has capacity $0$ in $\RR^3$, so that $u$ is equal to the Perron (or variational) solution $u_\phi$ by Theorem~\ref{thm:3.2}.
But approaching $z_0$ along the level sets $\Gamma_c$ for different levels $c$ with $V(0,0)\le c< B$ yields that $u=u_\phi$ is discontinuous in $z_0$.
\end{proof}

\begin{rem}
We use the notation from the above proof.
\begin{enumerate}[1.]
\item
It follows easily that exactly the values in $[V(0,0),B]$ are the possible limit points of $u_\phi(x)$ as $x\to z_0$.

\item
Regularity of $\phi$ on $\partial\Omega$ does not imply convergence of $u_\phi(x)$ as $x\to z_0$.
In the proof $\phi$ is constant on the two components of $\partial\Omega$.
So, in particular, $\phi=\restrict{\Phi}{\partial\Omega}$ for some $\Phi\in C^\infty_\cpt(\RR^d)$.

\item
In~\cite{Leb1913} the concepts of the Perron solution and singular points were not yet available.
So Lebesgue, in his setting, instead asserts that~\ref{eq:D} cannot have a classical solution $u_\phi$ on $\Omega$, since that one would need to be equal to the discontinuous function $u$.
This can be seen applying the strong maximum principle to
$x\mapsto u_\phi(x)-u(x)-\eps\frac{1}{\norm{x}_2}$ and $x\mapsto u(x)-u_\phi(x)-\eps\frac{1}{\norm{x}_2}$, which for all $\eps>0$ are harmonic functions on $\Omega$ with a negative limit superior towards all points in $\partial\Omega$.

\item
It follows from Wiener's criterion~\cite[p.~130]{Wie24}, see also Theorem~\ref{thm:wiener-crit} later on, that the cusp of $\Gamma_B$ at $z_0$ is very thin in the sense that there is only a vanishingly small amount of $\RR^3\setminus\Omega$ close to $z_0$ with respect to the Sobolev capacity. Let $r_B\colon [0,z_2]\to[0,\infty)$ be the contour function as in Proposition~\ref{prop:lebesgue-contours} of the surface of revolution $\Gamma_B$.
In the case when $r_B$ is monotonically increasing close to $0$, which for general $\rho$ cannot be expected, the singularity of $z_0$ can be characterised by the (eventual) convergence of the series $\sum_{j=j_0}^\infty \frac{1}{\log(r_B(q^j))}$ for one/every $q\in (0,1)$, see~\cite[p.~140]{Wie24} or~\cite[p.~287--288]{La72}.
In particular, for a domain in $\RR^3$ with an inward pointing cusp locally given as the surface of revolution about the $z$-axis of the monotonically increasing contour function $r\colon[0,\frac{1}{2}]\to[0,\infty)$ with $r(0)=0$, the tip $z_0$ will be singular for $r(z)=z^{-\log z}$ and regular for $r(z)=(-\log z)^{\log z}$, for example.
\end{enumerate}
\end{rem}

We give some estimates for the behaviour of the cusps of the level surfaces $\Gamma_c$ for $c>1$ towards $(0,0,0)$.
Our results rely on properties of the density function $\rho$.
We recall the notion of Dini continuity. Let $f\colon[a,b]\to\RR$ be continuous. Then the \emphdef{modulus of continuity} of $f$, which gives quantitative information about the uniform continuity of $f$, is the map $\omega_f\colon [0,\infty)\to[0,\infty)$ given by
\[
    \omega_f(t) = \sup\{\abs{f(x)-f(y)}: x,y\in [a,b]\text{ and }\abs{x-y}\le t\},
\]
and $f$ is called \emphdef{Dini continuous} if $\int_0^1\frac{\omega_f(t)}{t}\dx[t]<\infty$.
So if $\rho$ is Dini continuous locally at $0$, i.e.~on an interval $[0,b]$ with $b>0$, then in particular~\eqref{eq:lebesgue-potential-critical} is satisfied as $\rho(0)=0$.
Moreover, note that every Hölder continuous function is Dini continuous.
\begin{prop}\label{prop:cusp-rate}
Adopt the assumptions from Proposition~\ref{prop:lebesgue-contours}. Let $\alpha>0$.
\begin{alenum}
\item\label{en:cr-mon-est}
Suppose that $\rho$ is monotonically increasing locally at $0$. Let $\delta\in(0,1)$. Then $\lim_{z\to 0+}\rho(z)\log z = 0$ and
\begin{align*}
    \limsup_{z\to 0+} V\bigl(\exp\bigl(-\tfrac{\alpha}{\rho((1+\delta)z)}\bigr),z\bigr) &\le V(0,0)+2\alpha \\
        &\le\liminf_{z\to 0+} V\bigl(\exp\bigl(-\tfrac{\alpha}{\rho((1-\delta)z)}\bigr),z\bigr).
\end{align*}

\item\label{en:cr-dini-est}
Suppose that $\rho$ is Dini continuous locally at $0$.
Then  $\lim_{z\to 0+}\rho(z)\log z=0$ and
\[
    \lim_{z\to 0+} V(e^{-\alpha/\rho(z)}, z) = V(0,0)+2\alpha.
\]
\end{alenum}
\end{prop}
\begin{proof}
\ref{en:cr-mon-est}
For the first limit, it suffices to show that $\lim_{t\to\infty}t u(t)=0$ with $u(t)=\rho(e^{-t})$.
Note that there exists a $t_0>0$ such that $u$ is monotonically decreasing to $0$ on $[t_0,\infty)$.
By~\eqref{eq:lebesgue-potential-critical} one has
\[
    \int_{-\log L}^\infty u(t)\dx[t] = \int_0^L \frac{\rho(\zeta)}{\zeta}\dx[\zeta]<\infty.
\]
Suppose for contradiction that $\limsup_{t\to\infty}tu(t)\ge\gamma>0$. Then there exists a sequence $(t_n)_{n\in\NN}$ monotonically increasing to $\infty$ such that $t_n u(t_n)\ge\frac{\gamma}{2}$ and $t_{n}\ge 2 t_{n-1}$ for all $n\in\NN$. We obtain a contradiction by observing
\[
    \int_{-\log L}^\infty u(t)\dx[t] \ge \sum_{n=1}^\infty \int_{t_{n-1}}^{t_{n}}\frac{\gamma}{2 t_{n}} \ge\sum_{n=1}^\infty \frac{\gamma}{4}=\infty.
\]  

For the estimate on the left hand side, let $e(z)=\exp\bigl(-\alpha/\rho((1+\delta)z)\bigr)$. Suppose that $z\in(0,\frac{L}{2})$ is so small that $z^2+(e(z))^2\le\frac{1}{4}$ and that $\rho$ is monotonically increasing on $[0,2z]$.
We write
\begin{align*}
    V(e(z),z) &= \int_0^{(1+\delta)z}\frac{\rho(\zeta)}{\sqrt{(\zeta-z)^2+(e(z))^2}}\dx[\zeta] + \int_{(1+\delta)z}^{L}\frac{\rho(\zeta)}{\sqrt{(\zeta-z)^2+(e(z))^2}}\dx[\zeta] \\
              &= V_\low(z) + V_\high(z).
\end{align*}
Then for $\zeta\in [(1+\delta)z,L]$ one has $\abs{\zeta-z}=\zeta-z\ge\frac{\delta}{1+\delta}\zeta$, and therefore 
\[
   \frac{\rho(\zeta)}{\sqrt{(\zeta-z)^2+(e(z))^2}}\le \frac{\rho(\zeta)}{\zeta} \frac{1+\delta}{\delta}.
\]
As the function on the right hand side is integrable in $\zeta$ over $[0,L]$, we can apply dominated convergence and obtain
\[
    \lim_{z\to0+} V_\high(z) = \int_0^L\frac{\rho(\zeta)}{\zeta}\dx[\zeta] = V(0,0).
\]

For the first term, by monotonicity and explicit calculation it follows that
\begin{align*}
    V_\low(z)&\le \rho\p[\big]((1+\delta)z)\int_0^{(1+\delta)z}\frac{1}{\sqrt{(\zeta-z)^2+(e(z))^2}}\dx[\zeta] \\
        &= \rho\p[\big]((1+\delta)z)\p(\log(\sqrt{(\delta z)^2 + (e(z))^2}+\delta z) - \log(\sqrt{z^2+(e(z))^2}-z)) \\
        &= -\rho\p[\big]((1+\delta)z)\p(\log(\sqrt{z^2+(e(z))^2}-z)+\log(\sqrt{z^2+(e(z))^2}+z)) \\
        &\qquad {}+\rho\p[\big]((1+\delta)z)\p(\log(\sqrt{(\delta z)^2 + (e(z))^2}+\delta z) + \log(\sqrt{z^2+(e(z))^2}+z)) \\
        &\le -\rho\p[\big]((1+\delta)z)\log\p((e(z))^2) - 2\rho\p[\big]((1+\delta)z)\log(\delta z) \\
        &= 2\alpha - 2\rho\p[\big]((1+\delta)z)(\log((1+\delta)z) + \log\tfrac{\delta}{1+\delta})
\end{align*}
So as $\lim_{z\to 0+}\rho(z)\log z=0$, we obtain $\limsup_{z\to 0+} V_\low(z)\le 2\alpha$, which concludes the argument for the first estimate.

The estimate on the right hand side follows in an analogous way, where we instead let $e(z)=\exp\bigl(-\alpha/\rho((1-\delta)z)\bigr)$ and use that
\begin{align*}
    V_\low(z)&\ge \int_{(1-\delta)z}^{(1+\delta)z} \frac{\rho(\zeta)}{\sqrt{(\zeta-z)^2+(e(z))^2}}\dx[\zeta] \\
        &\ge\rho((1-\delta)z)\int_{(1-\delta)z}^{(1+\delta)z} \frac{1}{\sqrt{(\zeta-z)^2+(e(z))^2}} \\
        &=\rho((1-\delta)z)\p(\log(\sqrt{(\delta z)^2 + (e(z))^2}+\delta z) - \log(\sqrt{(\delta z)^2 + (e(z))^2}-\delta z)).
\end{align*}

\ref{en:cr-dini-est}
Assume that $\rho$ is Dini continuous on $[0,b]$ with $b\in (0,L]$. Let $\tilde{\rho}=\restrict{\rho}{[0,b]}$.

We first show that $\lim_{z\to0+}\rho(z)\log z=0$. Let $\eps>0$. Choose $\eta\in(0,\min\{b,1\}]$ small enough such that $\int_0^\eta\frac{\omega_{\tilde{\rho}}(t)}{t}\dx[t]<\eps$.
Suppose that $z\in (0,\eta]$.
Then
\begin{align*}
    \rho(z)\abs{\log z} \le\omega_{\tilde{\rho}}(z)\int_z^1\frac{1}{t}\dx[t] &\le \int_z^\eta\frac{\omega_{\tilde{\rho}}(t)}{t}\dx[t] + \omega_{\tilde{\rho}}(z)\int_\eta^1\frac{1}{t}\dx[t] \\
        &\le \int_0^\eta\frac{\omega_{\tilde{\rho}}(t)}{t}\dx[t] - \omega_{\tilde{\rho}}(z)\log\eta\le\eps - \omega_{\tilde{\rho}}(z)\log\eta,
\end{align*}
where we use that $\omega_{\tilde{\rho}}$ is monotonically increasing for the second inequality.
As $\lim_{z\to 0+}\omega_{\tilde{\rho}}(z)=0$ and $\eps>0$ was arbitrary, it follows that $\lim_{z\to0+}\rho(z)\log z=0$.

For the proof of the second limit, suppose that $z\in(0,\frac{b}{2}]$ and write
\begin{align*}
  V(r,z)&=\int_0^L\frac{\rho(\zeta)}{\sqrt{(\zeta-z)^2+r^2}}\dx[\zeta]\\
        &=\rho(z)\int_0^{2z}\frac{1}{\sqrt{(\zeta-z)^2+r^2}}\dx[\zeta]
          +\int_0^{2z}\frac{\rho(\zeta)-\rho(z)}{\sqrt{(\zeta-z)^2+r^2}}\dx[\zeta] \\
        &\qquad {}+ \int_{2z}^L \frac{\rho(\zeta)}{\sqrt{(\zeta-z)^2+r^2}}\dx[\zeta]\\
  &=V_1(r,z)+V_2(r,z) + V_3(r,z).
\end{align*}
Then
\begin{align*}
    V_2(r,z) 
        &=\int_0^z \frac{\rho(z-t)-\rho(z)}{\sqrt{t^2+r^2}}\dx[t] + \int_0^{z} \frac{\rho(z+t)-\rho(z)}{\sqrt{t^2+r^2}}\dx[t].
\end{align*}
Note that the terms inside of the two integrals are bounded by
\[
    \frac{\abs{\rho(z\pm t)-\rho(z)}}{\sqrt{t^2+r^2}} \leq\frac{\omega_{\tilde{\rho}}(t)}{t},
\]
which is integrable in $t$ in particular over $[0,\frac{b}{2}]$ by the Dini continuity of $\tilde{\rho}$.
So the dominated convergence theorem yields $\lim_{(r,z)\to(0,0)} V_2(r,z)=0$.

One can deal with $V_3$ in the same way as with $V_\high$ in the proof of~\ref{en:cr-mon-est}. It follows that
\[
    \lim_{(r,z)\to (0,0)} V_3(r,z)=\int_0^L \frac{\rho(\zeta)}{\zeta}\dx[\zeta]=V(0,0).
\]

Finally, we have the explicit description
\[
    V_1(r,z)=\rho(z)\log\left(\sqrt{z^2+r^2}+z\right)-\rho(z)\log\left(\sqrt{z^2+r^2}-z\right).
\]
Hence
\begin{align*}\MoveEqLeft
    \lim_{z\to 0+} V_1(e^{-\alpha/\rho(z)},z) = -\lim_{z\to 0+}\rho(z)\log(\sqrt{z^2+e^{-2\alpha/\rho(z)}} - z) \\
        &= -\lim_{z\to 0+}\rho(z)\p(\log(\sqrt{z^2+e^{-2\alpha/\rho(z)}} - z) + \log(\sqrt{z^2+e^{-2\alpha/\rho(z)}} + z))\\
        &= 2\alpha,
\end{align*}
where we use $\lim_{z\to 0+}\rho(z)\log z=0$ twice.
\end{proof}

The domains in Corollary~\ref{cor:sing-dom} all share the singular boundary point $(0,0,0)$ as the tip of an inward pointing cusp along the $z$-axis.
The boundary component with the cusp is given by the level surface $\Gamma_c\cup\{(0,0,0)\}$ for a $c>V(0,0)$.
It is natural to ask how the geometry of these cusps depends on the density function $\rho$.
The following result gives decay rates of the cusp's contour function in terms of $\rho$.
\begin{cor}\label{cor:cusp-rate}
Adopt the assumptions from Proposition~\ref{prop:lebesgue-contours}.
Let $c>V(0,0)$ and denote by $r_c\colon [0,z_2]\to[0,\infty)$ the contour function for $\Gamma_c$.
Let $0<\alpha<\frac{c-V(0,0)}{2}<\beta$.

\begin{alenum}
\item\label{en:cr-mon}
Suppose that $\rho$ is monotonically increasing locally at $0$. Let $\delta\in(0,1)$.
Then there exists a $z_0\in (0,L]$ such that
\[
    \exp\bigl(-\tfrac{\beta}{\rho((1-\delta)z)}\bigr)< r_c(z) < \exp\bigl(-\tfrac{\alpha}{\rho((1+\delta)z)}\bigr)
\]
for all $z\in (0,z_0]$.
\item\label{en:cr-dini}
Suppose that $\rho$ is Dini continuous locally at $0$.
Then there exists a $z_0\in(0,L]$ such that
\[
    e^{-\beta/\rho(z)} < r_c(z) < e^{-\alpha/\rho(z)}
\]
for all $z\in (0,z_0]$.
\end{alenum}
\end{cor}
\begin{proof}
We give the argument for the upper estimate in~\ref{en:cr-dini}. The other estimates follow in an analogous way.

Suppose for contradiction that there exists a sequence $(z_n)_{n\in\NN}$ in $(0,L]$ with $z_n\to 0$ and such that $r_c(z_n)\ge e^{-\alpha/\rho(z_n)}$ for all $n\in\NN$.
Then, as $r\mapsto V(r,z)$ is monotonically decreasing,
\[
    c = V(r_c(z_n),z_n)\le V(e^{-\alpha/\rho(z_n)},z_n)\to V(0,0)+2\alpha
\]
for $n\to\infty$ by Proposition~\ref{prop:cusp-rate}~\ref{en:cr-dini-est}, which is a contradiction as $V(0,0)+2\alpha<c$.
\end{proof}

For the remainder of the section we specialise to Lebesgue's case $L=1$ and $\rho(z)=z$ for $z\in[0,1]$.
Then it is easy to compute $V$ explicitly, namely
\begin{equation}
  \label{eq:lebesgue-potential-formula}
  \begin{aligned}
    V(r,z)
    &=\int_0^1\frac{\zeta}{\sqrt{(\zeta-z)^2+r^2}}\dx[\zeta]\\
    &=z\int_0^1\frac{1}{\sqrt{(\zeta-z)^2+r^2}}\dx[\zeta]+\int_0^1\frac{\zeta-z}{\sqrt{(\zeta-z)^2+r^2}}\dx[\zeta]\\
    &=
      \begin{multlined}[t]
        z\log\mleft(\sqrt{(1-z)^{2}+r^2}+1-z\mright)-z\log\mleft(\sqrt{r^{2} + z^{2}}-z\mright)\\
        + \sqrt{(1-z)^2 + r^{2}} - \sqrt{r^{2} + z^{2}}.
      \end{multlined}
  \end{aligned}
\end{equation}
Figure~\ref{fig:lebesgue-graph} shows the graph of $V$ as a function of $r>0$ and $z\in\RR$. One can follow the level curve $V(r,z)=c>1$ to get the limit $c>1$ at $(0,0)$. Figure~\ref{fig:lebesgue-contour} shows cross sections along the $xz$-plane of $\Gamma_c$. The level sets $\Gamma_{\frac{1}{2}}$ and $\Gamma_2$ bounding \emphdef{Lebesgue's domain} given by
\begin{displaymath}
  \Omega:=\mleft\{(x,y,z)\in D\setcolon \frac{1}{2}<V\mleft(\sqrt{x^2+y^2},z\mright)<2\mright\}
\end{displaymath}
are highlighted in red. The rod is shown in blue. Figure~\ref{fig:lebesgue-domain} shows the domain cut open, with the rod carrying the mass distribution shown in red. It is interesting to note that later authors such as~\cite[page~6]{Kel41} only use the part
\begin{displaymath}
  \widetilde{V}(r,z)=z\log\mleft(\sqrt{z^{2}+r^2}-z\mright)+\sqrt{z^2 + r^{2}}
\end{displaymath}
of~\eqref{eq:lebesgue-potential-formula}, which is harmonic in $\RR^3\setminus\{(0,0,z)\setcolon z\ge 0\}$, in order to give an example of a domain with a singular point attributed to an unidentified paper of Lebesgue, but without any reference of its physical significance.
\begin{figure}[ht]
  \centering
  \begin{tikzpicture}[background rectangle/.style={fill=white},show background rectangle,inner frame sep=0pt]
    \node[anchor=south west,inner sep=0] (cliff) at (0,0) %
    {\includegraphics[width=0.37\textwidth]{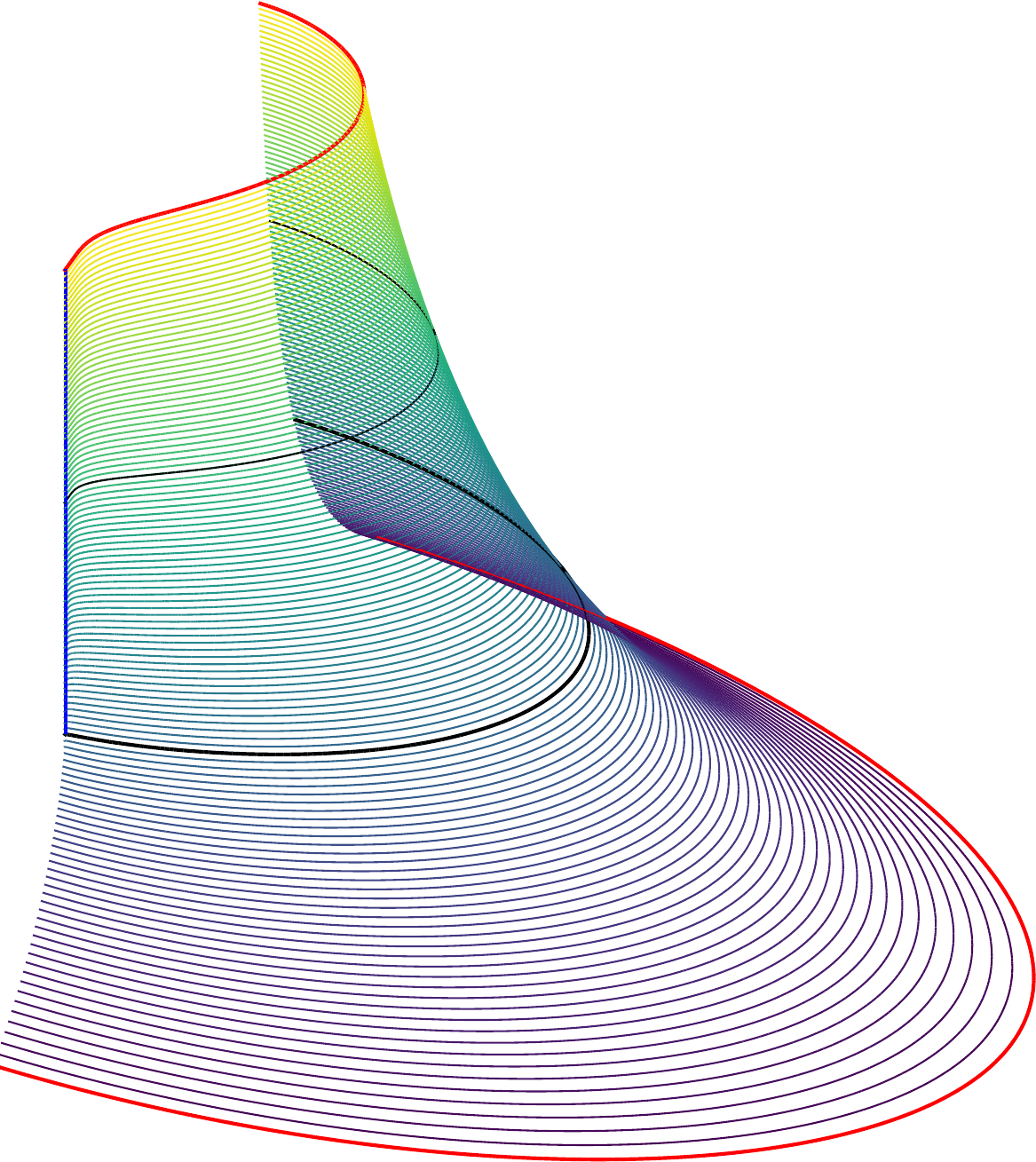}};%
    \begin{scope}[x={(cliff.south east)},y={(cliff.north west)}]
      \coordinate (r) at (-15.0:0.635);
      \coordinate (z) at (51.0:0.584);
      \coordinate (sing) at (0.061,0.368);
      \coordinate (c-unit) at (0,0.397);
      \coordinate (origin) at ($(sing)-(c-unit)$);
      \coordinate (dash) at (-0.02,0);
      \draw ($(origin)-0.01*(c-unit)$) -- ($(origin)+0.24*(c-unit)$);
      \draw[gray] ($(origin)+0.25*(c-unit)$) -- ($(origin)+(c-unit)$);
      \draw[->] ($(origin)+2*(c-unit)$) -- ($(origin)+2.35*(c-unit)$) node[right] {$c$};%
      \draw[thick,blue] (sing) -- ++(c-unit);
      \draw ($(sing)+(c-unit)$) -- ++(dash) node[left] {$2$};%
      \draw ($(sing)+0.5*(c-unit)$) -- ++(dash) node[left] {$\frac{3}{2}$};%
      \draw (sing) -- ++(dash) node[left] {$1$};%
      \draw[gray] ($(sing)-0.5*(c-unit)$) -- ++(dash) node[left,black] {$\tfrac{1}{2}$};%
      \node[left] at (origin) {$0$};
      \draw ($(origin)-0.3*(z)$) -- ($(origin)+0.175*(z)$);
      \draw[gray] ($(origin)+0.185*(z)$) -- ($(origin)+1.25*(z)$);
      \draw[->] ($(origin)+1.25*(z)$) -- ($(origin)+1.5*(z)$) node[right] {$z$};
      \draw[->] ($(origin)-0.03*(r)$) -- ++($1.2*(r)$) node[right] {$r$};
    \end{scope}
  \end{tikzpicture}
  \caption{Graph of $V(r,z)$ between the level sets $c=\frac{1}{2}$ and $c=2$. One clearly sees the singularity at $(r,z)=(0,0)$.}
  \label{fig:lebesgue-graph}
\end{figure}
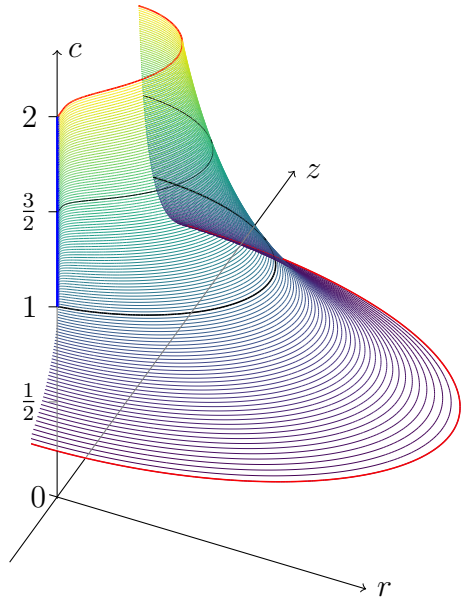
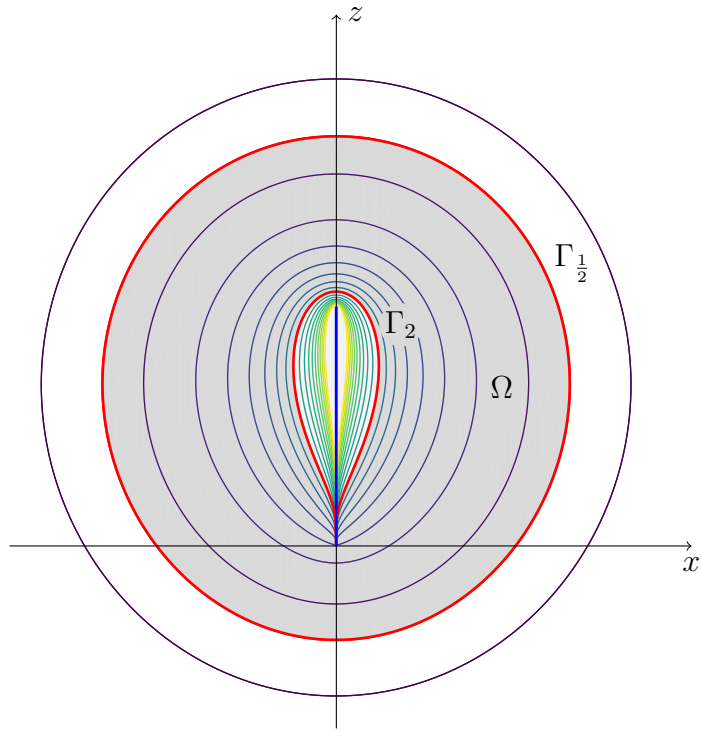
\begin{figure}[ht]
  \centering
  \begin{tikzpicture}[background rectangle/.style={fill=white},show background rectangle,inner frame sep=0pt]
    \node[anchor=south west,inner sep=0] (contour) at (0,0) %
    {\includegraphics[width=0.51\textwidth]{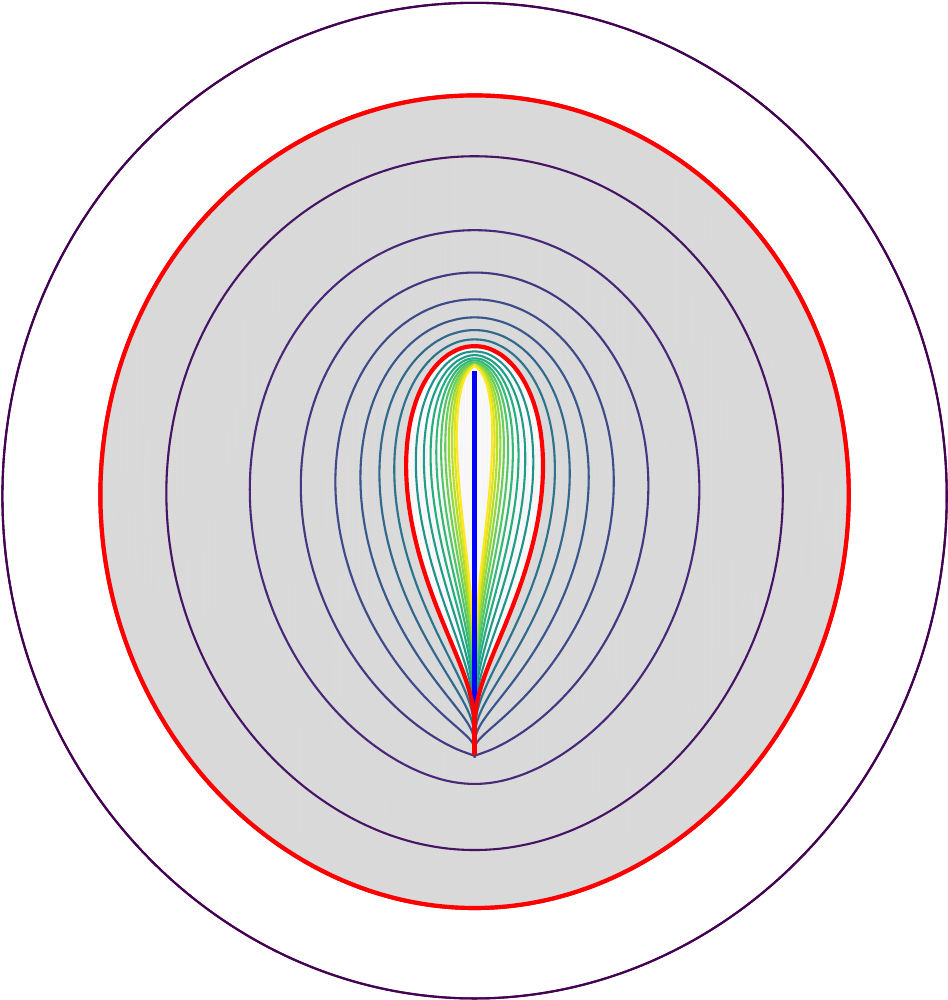}};%
    \begin{scope}[x={(contour.south east)},y={(contour.north west)}]
      \draw[->] (-0.05,0.244) -- (1.1,0.244) node[below] {$x$};%
      \draw[->] (0.501,-0.05) -- (0.501,1.1) node[right] {$z$};%
          \draw[blue,thick] (0.501,0.244) -- (0.501,0.63);%
      \node at (0.9,0.7) {$\Gamma_{\frac{1}{2}}$};
      \node[fill=black!15,circle,inner sep=0.25pt] at (0.61,0.6) {$\Gamma_{2}$};
      \node[fill=black!15,circle,inner sep=0.25pt] at (0.78,0.5) {$\Omega$};
    \end{scope}
  \end{tikzpicture}
  \caption{Contour map in the $xz$-plane showing $\Gamma_{\frac{1}{2}}$ and $\Gamma_2$ bounding Lebesgue's domain.}
  \label{fig:lebesgue-contour}
\end{figure}

\begin{figure}[ht]
  \centering
  \includegraphics[width=0.5\textwidth]{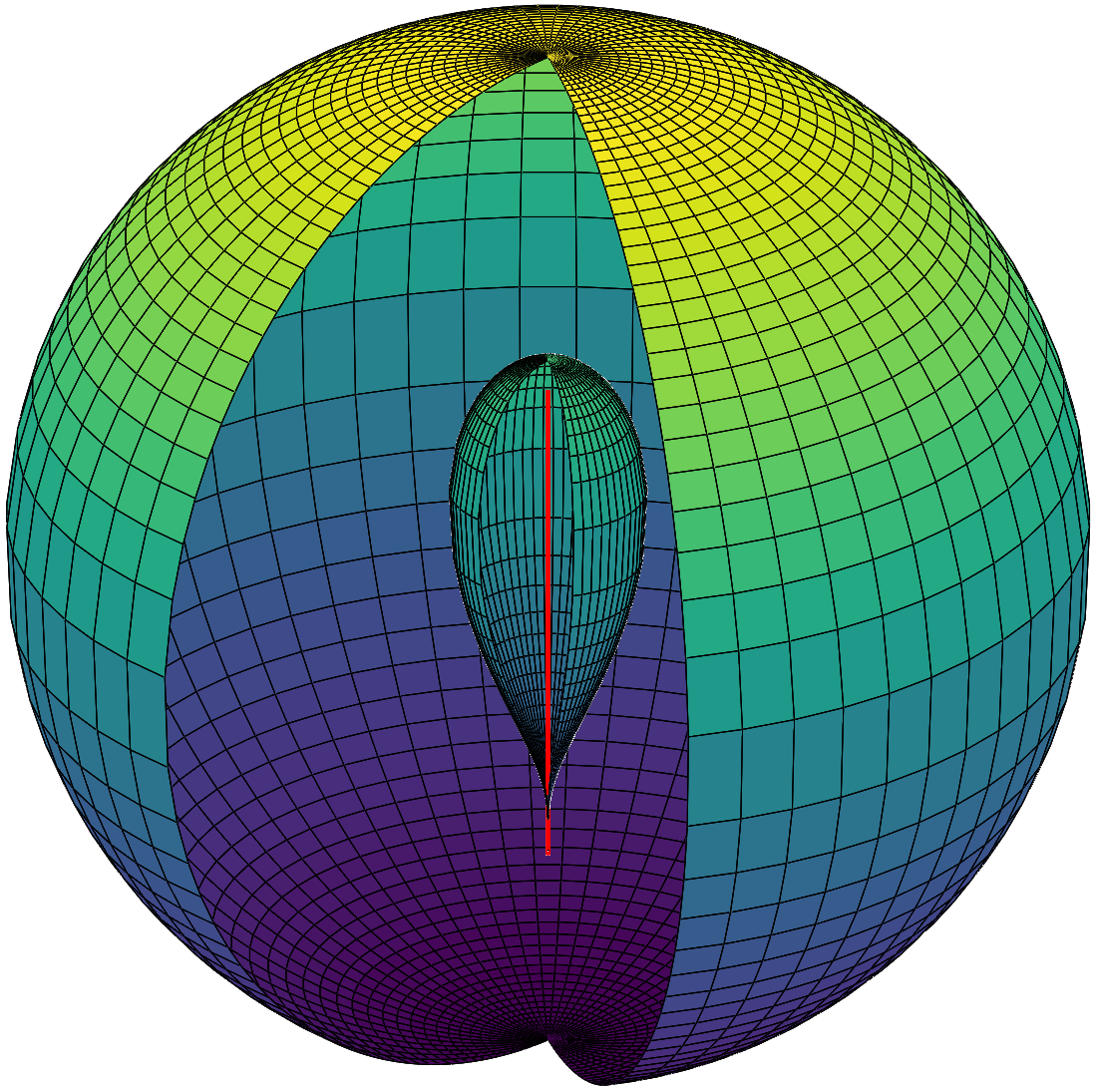}
  \caption{Lebesgue's domain cut open along its outer boundary $\Gamma_{\frac{1}{2}}$ to reveal the cusp of the inner boundary $\Gamma_{2}$.}
  \label{fig:lebesgue-domain}
\end{figure}

It follows easily from Corollary~\ref{cor:cusp-rate}~\ref{en:cr-dini} that the level sets $\Gamma_c$ for $c>1$ have exponential cusps towards $(0,0,0)$.
In fact, fix $c>1$ and let $r_c\colon[0,z_2]\to[0,\infty)$ be the corresponding contour function as in Proposition~\ref{prop:lebesgue-contours}.
Then for all $\alpha,\beta\in\RR$ with $0<\alpha < \frac{c-1}{2}<\beta$ there exists a $z_0\in (0,1]$ such that
\[
   e^{-\beta/z}<r_c(z)< e^{-\alpha/z}
\]
for all $z\in(0,z_0]$.
So specifically for Lebesgue's domain with $c=2$, we obtain that the cusp at $(0,0,0)$ is eventually thinner than $e^{-\alpha/z}$ as $z\to 0+$ for every $\alpha\in (0,\frac{1}{2})$,
and eventually thicker than $e^{-\alpha/z}$ as $z\to 0+$ for every $\alpha>\frac{1}{2}$. We point out that consequently also Lebesgue's irregularity condition~\cite[p.~352]{Leb24} for a singular point at $(0,0,0)$ is clearly satisfied. Note that a merely polynomial decay of the contour function as $z\to 0+$ would imply regularity by~\cite[point~c.~on p.~352]{Leb24}.

The presence of a singular boundary point introduces some surprisingly strong non-local behaviour.
We will illustrate this for the example where $\Omega$ is Lebesgue's domain.
We again go back to writing $x$ for points in $\Omega$ and $z$ for points on $\partial\Omega$.
Recall that the boundary $\partial\Omega$ of $\Omega$ is the disjoint union of the two closed sets $\Gamma_{\frac{1}{2}}$ and $\Gamma_2\cup\{z_0\}$ with the singular point $z_0=(0,0,0)$.

If $\phi=\frac{1}{2}$ on $\Gamma_{\frac{1}{2}}$ and $\phi=2$ on $\Gamma_2\cup\{z_0\}$, then we know that $\lim_{x\to z_0} u_\phi(x)$ does not exist.
In fact, for each $c\in[1,2]$ there exists a sequence $(x_n)_{n\in\NN}$ in $\Omega$ such that $\lim_{n\to\infty}u_\phi(x_n)=c$.
Using our variational description of $u_\phi$, we now show the following.
Let $\alpha,\beta\in\RR$, $\alpha\ne\beta$.
Define $\psi\in C(\partial\Omega)$ by $\psi(z)=\alpha$ if $z\in \Gamma_{\frac{1}{2}}$ and $\psi(z)=\beta$ if $z\in\Gamma_2$.
Then $\lim_{x\to z_0} u_\psi(x)$ does not exist.
But $u_\psi\in H^1(\Omega)$.
\begin{proof}
Let $\eta\in C^\infty_\cpt(\RR^d)$ be such that $0\le\eta\le 1$, $\eta=0$ in a neighbourhood of $\Gamma_{\frac{1}{2}}$ and $\eta=1$ in a neighbourhood of $\Gamma_2$.
Let $\Phi_0=2\eta + \frac{1}{2}(1-\eta)$, $v_0\in H^1_0(\Omega)$ be such that $\Delta v_0=\Delta\Phi_0 = \frac{3}{2}\Delta\eta$ and $u_\phi = \Phi_0-v_0$.
Let $\Phi=\beta\eta + \alpha(1-\eta)$, $v\in H^1_0(\Omega)$ be such that $\Delta v=\Delta\Phi$.
Then $u_\psi = \Phi-v$.
Since $\Delta\Phi = (\beta-\alpha)\Delta\eta = \frac{2}{3}(\beta-\alpha)\Delta\Phi_0$, it follows that $v=\frac{2}{3}(\beta-\alpha)v_0$.
Hence $u_\psi=\Phi-\frac{2}{3}(\beta-\alpha)v_0$.
As $v_0(x)=\Phi_0(x)-u_\phi(x)$ does not converge as $x\to z_0$, the same is true for $u_\psi(x)$.
Since $\Phi\in C(\clos{\Omega})\cap H^1(\Omega)$, it follows that $u_\psi\in H^1(\Omega)$.
\end{proof}

The example shows that continuity at $z_0$ is a non-local property with respect to the boundary data.
Suppose that $\varphi\in C(\partial\Omega)$ admits a classical solution $u_\varphi$ on Lebesgue's domain $\Omega$. Then for $\psi=\eps\one_{\Gamma_{\frac{1}{2}}}$, one has that
\begin{displaymath}
  \lim_{x\to z_0}u_{\varphi+\psi}(x)=u_\phi(z_0) +\lim_{x\to z_0}u_\psi(x)
\end{displaymath}
does not exist and therefore $u_{\phi+\psi}\notin C(\clos{\Omega})$ whenever $\eps\ne 0$.
Note that $z_0$ is far away from $\Gamma_{\frac{1}{2}}$ and that the perturbation $\psi$ for the boundary data satisfies $\norm{\psi}_\infty\le\varepsilon$. In particular this shows that the set of $\phi\in C(\partial\Omega)$ such that~\ref{eq:D} has a classical solution has empty interior. By the maximum principle that set is also closed, showing that it is a meagre set in $C(\partial\Omega)$. In particular, the boundary data where the solution is discontinuous at $z_0$ is generic. We will show in the next section that in the presence of a singular point this is always the case, and that an even stronger non-local property holds.

\section{Non-locality}
\label{sec:5}
Our aim in this section is to prove the following result.
\begin{thm}\label{thm:5.1}
Suppose that $\Omega$ is connected.
Let $z_0\in \partial\Omega$ be a singular point.
Let $0\le\phi\in C(\partial\Omega)$ be such that $\phi(z_0)=0$ and such that $\phi(w)>0$ at a regular point $w\in \partial\Omega$. Then $u_\phi(x)$ does not converge to $0$ as $x\to z_0$.
\end{thm}

In order to illustrate the statement of Theorem~\ref{thm:5.1}, let $\Omega$ be once again Lebesgue's domain. For $\phi_0\equiv 0$, one has $u_{\phi_0}\equiv 0$, of course.
If we obtain $\phi$ by perturbing $\phi_0$ by a small positive bump somewhere away from $z_0$, regardless whether supported on $\Gamma_{\frac{1}{2}}$ or $\Gamma_2$, then $u_\phi(x)$ does no longer converge to $0$ as $x\to z_0$. This sharpens the example at the end of Section~\ref{sec:4} and shows that continuity of $u_\phi$ at $z_0$ with the proper value $\phi(z_0)$ is an unstable and also entirely non-local property.

As a simple corollary we obtain that in the presence of a singular boundary point it is generic for $\phi\in C(\partial\Omega)$, in the sense of the Baire category theorem, to not admit a classical solution.
\begin{cor}\label{cor:baire}
Suppose that $\Omega$ is not Dirichlet regular. Let
\[
    \mathcal{G} := \{\phi\in C(\partial\Omega) \setcolon \text{$u_\phi$ is classical solution of~\ref{eq:D}}\}.
\]
Then $\mathcal{G}$ is meagre in $C(\partial\Omega)$ and
therefore $C(\partial\Omega)\setminus\mathcal{G}$ is of second Baire category.
\end{cor}
\begin{proof}
By the maximum principle $\mathcal{G}$ is closed in $C(\partial\Omega)$.
Let $z_0\in\partial\Omega$ be singular.
By Kellogg's theorem and since $H^1(\Omega)\ne H^1_0(\Omega)$, there also exists a regular point $w\in \partial\Omega$.
Let $\delta = \frac{1}{2}\abs{w-z_0}$ and choose an $\eta\in C^\infty_\cpt(\RR^d)$ with $0\le\eta\le 1$, $\operatorname{supp}\eta\subset B(w,\delta)$ and $\eta(w)=1$.

Let $\phi_0\in\mathcal{G}$ and $\eps>0$. Consider $\psi = \phi_0 + \eps\restrict{\eta}{\partial\Omega}$.
Theorem~\ref{thm:5.1} and linearity of $\phi\mapsto u_\phi$ imply that $\psi\notin\mathcal{G}$. This shows that $\phi_0$ is not an inner point of $\mathcal{G}$ in $C(\partial\Omega)$.
So $\mathcal{G}$ is meagre and the remaining claim follows as $C(\partial\Omega)$ is a Baire space.
\end{proof}

\begin{rem}
It is possible to sharpen the assertions of Theorem~\ref{thm:5.1} and Corollary~\ref{cor:baire}.
In fact, singular points can be further distinguished into the semiregular and the strongly irregular singular points, see~\cite{LM81} and~\cite{BH22}. More precisely,
suppose that $\Omega$ is a bounded open domain with a singular boundary point $z_0$.
Then $z_0$ is \emphdef{semiregular} if $\lim_{x\to z_0} u_\phi(x)$ exists for all $\phi\in C(\partial\Omega)$, and \emphdef{strongly irregular} if for all $\phi\in C(\partial\Omega)$ there exists some sequence $(x_n)$ in $\Omega$ with $\lim_{n\to\infty} x_n=z_0$ such that $\lim_{n\to\infty} u_\phi(x_n)=\phi(z_0)$.
Somewhat surprisingly every singular point is of one of these two types by~\cite[Theorem~16 and Corollary~17]{LM81} or~\cite[Theorem~1.1]{BH22}.

It is easy to see that the origin is a semiregular singular point of the punctured disk $\{(x,y)\in\RR^2\setcolon 0<x^2+y^2<1\}$, while $z_0=(0,0,0)$ is a strongly irregular singular point for Lebesgue's domain.

If $z_0$ is a strongly irregular singular point in Theorem~\ref{thm:5.1}, one can sharpen the conclusion there to `$\lim_{x\to z_0}u_\phi(x)$ does not exist'.
So if $\Omega$ in Corollary~\ref{cor:baire} has a strongly irregular singular point, even
\[
    \widetilde{\mathcal{G}} := \{\phi\in C(\partial\Omega) \setcolon u_\phi \in C(\clos{\Omega})\}
\]
is meagre in $C(\partial\Omega)$.
\end{rem}

Theorem~\ref{thm:5.1} is contained in Landkof~\cite[p.\,243, lines 1--5]{La72}.
We will give a very different proof based on the following result. We assume that $d\ge 3$.
\begin{thm}\label{thm:5.2}
Let $z_0\in \partial\Omega$ be a singular point of $\Omega$.
Then there exists an open set $\Omega_0\subset\RR^d$ such that $\clos{\Omega_0}\subset\Omega\cup\{z_0\}$, $z_0\in\partial\Omega_0$ and $z_0$ is a singular point of $\Omega_0$.
\end{thm}
We postpone the proof of Theorem~\ref{thm:5.2} but note that in the case of Lebesgue's domain it readily follows from its construction.
More precisely, given the potential $V\colon D\to(0,\infty)$ from Lemma~\ref{lem:lebesgue-potential-basics}, we had $\Omega=\{x\in D\setcolon \frac{1}{2}<V(x)<2\}$. We can choose e.g.~$\Omega_0=\{x\in D\setcolon \frac{1}{c}<V(r,z)<c\}$ for any $c\in (1,2)$.
Now we can easily prove Theorem~\ref{thm:5.1} using the following criterion for a boundary point to be regular, see~\cite[Theorem~6.74]{AU23}.
\begin{prop}\label{prop:5.4}
Let $z_0\in \partial\Omega$. Then $z_0$ is regular if there exists a \emphdef{global $H^1$-barrier} $b$ at $z_0$; i.e.\ a function $b\in C(\clos{\Omega})\cap H^1(\Omega)$ such that $b(x)>0$ for all $x\in\clos{\Omega}\setminus\{z_0\}$, $b(z_0)=0$ and $-\Delta b\ge 0$ weakly on $\Omega$.
\end{prop}
\begin{proof}[Proof of Theorem~\ref{thm:5.1}]
Let $z_0\in \partial\Omega$ be a singular point for $\Omega$.
By Theorem~\ref{thm:5.2} there exists an open set $\Omega_0\subset\RR^d$ such that $\clos{\Omega_0}\subset\Omega\cup\{z_0\}$ and $z_0\in\partial\Omega_0$ is a singular point with respect to $\Omega_0$.
Let $0\le\phi\in C(\partial\Omega)$ be such that $\phi(z_0)=0$ and $\phi(w)>0$ at a regular point $w\in \partial\Omega$.
Then there exists $0\le\Phi_0\in C^\infty_\cpt(\RR^d)$ such that $\phi_0:=\restrict{\Phi_0}{\partial\Omega}$ satisfies $\phi_0\le\phi$ and $\phi_0(w)>0$.
Let $u_\phi$ and $u_{\phi_0}$ be the variational (equivalently Perron) solutions with respect to $\Omega$.
Then $u_\phi\ge u_{\phi_0}\ge 0$ on $\Omega$ by the maximum principle, and $u_{\phi_0}\ne 0$.
Moreover, $u_{\phi_0}\in H^1(\Omega)$ by Theorem~\ref{thm:2.7}.

Assume now for contradiction that $\lim_{x\to z_0}u_\phi(x)=0$. Let $b := \restrict{u_{\phi_0}}{\Omega_0}$. Then $\lim_{x\to z_0}b(x)=0$ and after continuous extension $b\in C(\clos{\Omega_0})\cap H^1(\Omega_0)$.
Let $x\in\clos{\Omega_0}\setminus\{z_0\}$. Then $b(x)=u_{\phi_0}(x)>0$ by the strong maximum principle~\cite[Section~2.2, Theorem~4]{Eva98} applied to $u_{\phi_0}$ on a sufficiently large subdomain compactly contained in $\Omega$.
So $b$ is a global $H^1$-barrier at $z_0$ with respect to $\Omega_0$.
Thus $z_0$ is regular with respect to $\Omega_0$ by Proposition~\ref{prop:5.4}, a contradiction.
\end{proof}

In the following proof of Theorem~\ref{thm:5.2} we adapt arguments of Björn~\cite[Proof of Theorem~1.1]{Bj07} and~\cite[Proof of Lemma~12.11]{HKM93}.
We recall Wiener's criterion, see e.g.~\cite[Theorem~6.3.3]{AH96}.
\begin{thm}[Wiener]\label{thm:wiener-crit}
Assume $d\ge 3$. Let $z_0\in \partial\Omega$. Then $z_0$ is regular if and only if
\[
    \sum_{j=0}^\infty 2^{j(d-2)}\capacity(\clos{B_j}\setminus\Omega)=\infty,
\]
where $B_j=B(z_0,2^{-j})$.
\end{thm}

\begin{proof}[Proof of Theorem~\ref{thm:5.2}]
Let $z_0\in \partial\Omega$ be a singular point, $B_j := B(z_0,2^{-j})$ for all $j\in\NN_0$ and $E=\RR^d\setminus\Omega$, where $d\ge 3$.
Then $E$ \emphdef{is thin at} $z_0$; i.e.
\[
    \sum_{j=0}^\infty 2^{j(d-2)}\capacity(E\cap\clos{B_j})<\infty.
\]
Fix $\eps_j>0$ for all $j\in\NN_0$ such that $\sum_{j=0}^\infty \eps_j 2^{j(d-2)}<\infty$.
Let $U_0=\RR^d$ and choose for $j=1,2,3,\dots$ iteratively open sets $U_j$ satisfying
\[
    E\cap\clos{B_j}\subset U_j\subset  B_{j-1}\cap U_{j-1}
\]
and
\[
    \capacity(U_j)\le\capacity(E\cap\clos{B_j})+\eps_j.
\]
Let $U = \bigcup_{j=0}^\infty U_j\setminus\clos{B_{j+1}}$.
Then $U$ is open and $z_0\notin U$.
Moreover, $U$ is thin at $z_0$ since
\[
    U\cap B_n = \bigcup_{j\ge n}(U_j\setminus\clos{B_{j+1}})\cap B_n\subset U_n.
\]
Observe that
\begin{align*}
    E\setminus\{z_0\} &= E\setminus \clos{B_1}\cup\bigcup_{j=1}^\infty (E\cap\clos{B_j})\setminus\clos{B_{j+1}} \\
        &\subset \bigcup_{j=0}^\infty U_j\setminus \clos{B_{j+1}} = U.
\end{align*}
Thus $\RR^d\setminus\Omega=E\subset U\cup \{z_0\}$ and $\RR^d\setminus U\subset \Omega\cup\{z_0\}$.
Now let
\[
    \Omega_0 := \{x\in \Omega \setcolon \dist(x,E)>\dist(x,\RR^d\setminus U)\}.
\]
Then $\Omega_0$ is open and $\Omega\setminus U\subset\Omega_0$ so that
\begin{equation}\label{eq:5.1}
    \RR^d\setminus\Omega_0\subset E\cup U = U\cup\{z_0\}.
\end{equation}
Thus $\RR^d\setminus\Omega_0$ is thin at $z_0$.
We show that $z_0\in \partial\Omega_0$.
Since there exists $C_1\ge 1$ such that $\frac{1}{C_1}r^{d-2}\le\capacity(\clos{B(z_0,r)})\le C_1r^{d-2}$ for all $r\in(0,1]$ by~\cite[Propositions~5.1.2 and~5.1.4]{AH96},
\begin{align*}
    \infty &= \sum_{j=1}^\infty 2^{j(d-2)}\capacity(\clos{B_j}) \\
        &\le\sum_{j=1}^\infty 2^{j(d-2)}\capacity(\clos{B_j}\cap U) + \sum_{j=1}^\infty 2^{j(d-2)}\capacity(\clos{B_j}\setminus U).
\end{align*}
Hence
\[
    \sum_{j=1}^\infty 2^{j(d-2)}\capacity(\clos{B_j}\setminus U) = \infty.
\]
Thus we find $x_j\in\clos{B_j}\setminus U$ with $x_j\ne z_0$ for all $j\in\NN$.
Then $\lim_{j\to\infty} x_j=z_0$.
Since $\RR^d\setminus U \subset\Omega_0\cup\{z_0\}$ by~\eqref{eq:5.1}, we have $x_j\in \Omega_0$ for all $j\in\NN$ and therefore $z_0\in\clos{\Omega_0}$.
Since $z_0\notin\Omega_0$ by the definition of $\Omega_0$, it follows that $z_0\in\partial\Omega_0$.

We claim that $\clos{\Omega_0}\subset\Omega\cup\{z_0\}$.
Let $x_n\in\Omega_0$ be such that $\lim_{n\to\infty} x_n = x \ne z_0$.
Then there exists a ball $B=B(z_0,\eps)$ such that $x_n\notin B$ for large $n$.
Then
\begin{equation}\label{eq:5.2}
    \dist(E\setminus B, \RR^d\setminus U)>0.
\end{equation}
In fact, $\RR^d\setminus U$ is compact, $E\setminus B$ closed and $(\RR^d\setminus U)\cap (E\setminus B)=\emptyset$ as $E\setminus B\subset U$.
Suppose for contradiction that $x\notin\Omega$. So $x\in E\setminus B$. Since $\dist(x_n,E)>\dist(x_n,\RR^d\setminus U)$, it follows that $\dist(x,\RR^d\setminus U)=0$.
This contradicts~\eqref{eq:5.2}.
\end{proof}

\def\cprime{$'$}


\begin{thebibliography}{DPRS23}

\bibitem[AH96]{AH96}
D.R. Adams and L.I. Hedberg, \emph{Function spaces and potential theory},
  Grundlehren der mathematischen Wissenschaften, vol. 314, Springer-Verlag,
  Berlin, 1996.
  DOI:\,\href{https://doi.org/10.1007/978-3-662-03282-4}{\nolinkurl{10.1007/978-3-662-03282-4}}

\bibitem[ADM21]{ADM21}
A.~Anop, R.~Denk, and A.~Murach, \emph{Elliptic problems with rough boundary
  data in generalized {S}obolev spaces}, Commun. Pure Appl. Anal. \textbf{20}
  (2021), 697--735.
  DOI:\,\href{https://doi.org/10.3934/cpaa.2020286}{\nolinkurl{10.3934/cpaa.2020286}}

\bibitem[AB92]{AB92}
W.~Arendt and P.~B\'enilan, \emph{Inegalit\'es de {K}ato et semi-groupes
  sous-markoviens}, Rev. Mat. Univ. Complut. Madrid \textbf{5} (1992),
  279--308.

\bibitem[AD08]{AD08:var}
W.~Arendt and D.~Daners, \emph{The {D}irichlet problem by variational methods},
  Bull. Lond. Math. Soc. \textbf{40} (2008), 51--56.
  DOI:\,\href{https://doi.org/10.1112/blms/bdm091}{\nolinkurl{10.1112/blms/bdm091}}

\bibitem[AtE19a]{AtE19:dp}
W.~Arendt and A.F.M. ter Elst, \emph{The {D}irichlet problem without the
  maximum principle}, Ann. Inst. Fourier (Grenoble) \textbf{69} (2019),
  763--782.
  DOI:\,\href{https://doi.org/10.5802/aif.3257}{\nolinkurl{10.5802/aif.3257}}

\bibitem[AtE19b]{AtE19:kato}
W.~Arendt and A.F.M. ter Elst, \emph{Kato's inequality}, Analysis and operator
  theory, Springer Optim. Appl., vol. 146, Springer, Cham, 2019, 47--60.

\bibitem[AtES24]{AtES24}
W.~Arendt, A.F.M. ter Elst, and M.~Sauter, \emph{The {P}erron solution for
  elliptic equations without the maximum principle}, Math. Ann. \textbf{390}
  (2024), 763--810.
  DOI:\,\href{https://doi.org/10.1007/s00208-023-02761-0}{\nolinkurl{10.1007/s00208-023-02761-0}}

\bibitem[AU23]{AU23}
W.~Arendt and K.~Urban, \emph{Partial differential equations---an introduction
  to analytical and numerical methods}, Graduate Texts in Mathematics, vol.
  294, Springer, Cham, 2023.
  DOI:\,\href{https://doi.org/10.1007/978-3-031-13379-4}{\nolinkurl{10.1007/978-3-031-13379-4}}

\bibitem[Bj{\"o}07]{Bj07}
A.~Bj{\"o}rn, \emph{Weak barriers in nonlinear potential theory}, Potential
  Anal. \textbf{27} (2007), 381--387.
  DOI:\,\href{https://doi.org/10.1007/s11118-007-9064-2}{\nolinkurl{10.1007/s11118-007-9064-2}}

\bibitem[BH22]{BH22}
A.~Bj\"orn and D.~Hansevi, \emph{Semiregular and strongly irregular boundary
  points for {$p$}-harmonic functions on unbounded sets in metric spaces},
  Collect. Math. \textbf{73} (2022), 253--270.
  DOI:\,\href{https://doi.org/10.1007/s13348-021-00317-6}{\nolinkurl{10.1007/s13348-021-00317-6}}

\bibitem[BP04]{BP04}
H.~Brezis and A.C. Ponce, \emph{Kato's inequality when {$\Delta u$} is a
  measure}, C. R. Math. Acad. Sci. Paris \textbf{338} (2004), 599--604.
  DOI:\,\href{https://doi.org/10.1016/j.crma.2003.12.032}{\nolinkurl{10.1016/j.crma.2003.12.032}}

\bibitem[Con78]{Con78}
J.B. Conway, \emph{Functions of one complex variable}, second ed., Graduate
  Texts in Mathematics, vol.~11, Springer-Verlag, New York-Berlin, 1978.
  DOI:\,\href{https://doi.org/10.1007/978-1-4612-6313-5}{\nolinkurl{10.1007/978-1-4612-6313-5}}

\bibitem[DL90]{DL90:vol1}
R.~Dautray and J.L. Lions, \emph{Mathematical analysis and numerical methods
  for science and technology. {V}ol. 1. {P}hysical origins and classical
  methods}, Springer-Verlag, Berlin, 1990.

\bibitem[DPRS23]{DPRS23}
R.~Denk, D.~Plo\ss, S.~Rau, and J.~Seiler, \emph{Boundary value problems with
  rough boundary data}, J. Differential Equations \textbf{366} (2023), 85--131.
  DOI:\,\href{https://doi.org/10.1016/j.jde.2023.04.001}{\nolinkurl{10.1016/j.jde.2023.04.001}}

\bibitem[Eva98]{Eva98}
L.C. Evans, \emph{Partial differential equations}, Graduate Studies in
  Mathematics, vol.~19, American Mathematical Society, Providence, RI, 1998.
  DOI:\,\href{https://doi.org/10.1090/gsm/019}{\nolinkurl{10.1090/gsm/019}}

\bibitem[EG15]{EG2015}
L.C. Evans and R.F. Gariepy, \emph{Measure theory and fine properties of
  functions}, revised ed., Textbooks in Mathematics, CRC Press, Boca Raton, FL,
  2015. DOI:\,\href{https://doi.org/10.1201/b18333}{\nolinkurl{10.1201/b18333}}

\bibitem[GT01]{GT01}
D.~Gilbarg and N.S. Trudinger, \emph{Elliptic partial differential equations of
  second order}, Classics in Mathematics, Springer-Verlag, Berlin, 2001.

\bibitem[Had06]{Had06}
J.~Hadamard, \emph{Sur le principe de {D}irichlet}, Bull. Soc. Math. Fr.
  \textbf{34} (1906), 135--138.
  DOI:\,\href{https://doi.org/10.24033/bsmf.774}{\nolinkurl{10.24033/bsmf.774}}

\bibitem[HKM93]{HKM93}
J.~Heinonen, T.~Kilpel\"ainen, and O.~Martio, \emph{Nonlinear potential theory
  of degenerate elliptic equations}, Oxford Mathematical Monographs, The
  Clarendon Press, Oxford University Press, New York, 1993.

\bibitem[Kel41]{Kel41}
M.V. Keldysh, \emph{On the solubility and the stability of {D}irichlet's
  problem}, Uspekhi Matem. Nauk \textbf{8} (1941), 171--231, English
  Translation: Amer. Math. Soc. Translations (2) \textbf{51} (1966), 1--73.

\bibitem[Lan72]{La72}
N.S. Landkof, \emph{Foundations of modern potential theory}, Die Grundlehren
  der mathematischen Wissenschaften, vol. 180, Springer-Verlag, New
  York-Heidelberg, 1972.

\bibitem[Leb24]{Leb24}
H.~Lebesgue, \emph{Conditions de r{\'e}gularit{\'e}, conditions
  d'irregularit{\'e}, conditions d'impossibilit{\'e}s dans le probl{\`e}me de
  {Dirichlet}.}, C. R. Acad. Sci., Paris \textbf{178} (1924), 349--354
  (French). Available at
  \href{http://gallica.bnf.fr/ark:/12148/bpt6k3131z.f349}{\nolinkurl{http://gallica.bnf.fr/ark:/12148/bpt6k3131z.f349}}

\bibitem[Leb13]{Leb1913}
M.H. Lebesgue, \emph{Sur des cas d'impossibilit{\'e} du probl{\`e}me de
  {D}irichlet ordinaire}, Comptes-rendus des s{\'e}ances / Soci{\'e}t{\'e}
  math{\'e}matique de France, Gauthier-Villars, Paris, 1913, 17. Available at
  \href{https://gallica.bnf.fr/ark:/12148/bpt6k9446739}{\nolinkurl{https://gallica.bnf.fr/ark:/12148/bpt6k9446739}}

\bibitem[LM81]{LM81}
J.~Luke{\v s} and J.~Mal{\'y}, \emph{On the boundary behaviour of the {P}erron
  generalized solution}, Math. Ann. \textbf{257} (1981), 355--366.
  DOI:\,\href{https://doi.org/10.1007/BF01456505}{\nolinkurl{10.1007/BF01456505}}

\bibitem[MS98]{MS98}
V.~Maz{\cprime}ya and T.~Shaposhnikova, \emph{Jacques {H}adamard, a universal
  mathematician}, History of Mathematics, vol.~14, American Mathematical
  Society, Providence, RI; London Mathematical Society, London, 1998.
  DOI:\,\href{https://doi.org/10.1090/hmath/014}{\nolinkurl{10.1090/hmath/014}}

\bibitem[McS34]{McS34}
E.J. McShane, \emph{Extension of range of functions}, Bull. Amer. Math. Soc.
  \textbf{40} (1934), 837--842.
  DOI:\,\href{https://doi.org/10.1090/S0002-9904-1934-05978-0}{\nolinkurl{10.1090/S0002-9904-1934-05978-0}}

\bibitem[Per23]{Per1923}
O.~Perron, \emph{Eine neue {B}ehandlung der ersten {R}andwertaufgabe f\"ur
  {$\Delta u=0$}}, Math. Z. \textbf{18} (1923), 42--54.
  DOI:\,\href{https://doi.org/10.1007/BF01192395}{\nolinkurl{10.1007/BF01192395}}

\bibitem[Pry71]{Pry71}
F.E. Prym, \emph{On the integration of the differential equation
  {$\frac{\partial^2 u}{\partial x^2}+\frac{\partial^2 u}{\partial y^2}=0$}},
  J. Reine Angew. Math. \textbf{73} (1871), 340--364.
  DOI:\,\href{https://doi.org/10.1515/crll.1871.73.340}{\nolinkurl{10.1515/crll.1871.73.340}}

\bibitem[Ran95]{Ran95}
T.~Ransford, \emph{Potential theory in the complex plane}, London Mathematical
  Society Student Texts, vol.~28, Cambridge University Press, Cambridge, 1995.
  DOI:\,\href{https://doi.org/10.1017/CBO9780511623776}{\nolinkurl{10.1017/CBO9780511623776}}

\bibitem[Sch66]{Sch1966}
L.~Schwartz, \emph{Th\'eorie des distributions}, Publications de l'Institut de
  Math\'ematique de l'Universit\'e{} de Strasbourg, vol. IX-X, Hermann, Paris,
  1966.

\bibitem[Vas35]{Vas35}
F.~Vasilesco, \emph{Sur la m{\'e}thode du balayage de {Poincar{\'e}}
  {\'e}tendue par {M}.~de {La} {Vall{\'e}e} {Poussin}, et ses rapports avec le
  probl{\`e}me de {Dirichlet} g{\'e}n{\'e}ralis{\'e}}, C. R. Acad. Sci., Paris
  \textbf{200} (1935), 199--201. Available at
  \href{https://gallica.bnf.fr/ark:/12148/bpt6k3152t}{\nolinkurl{https://gallica.bnf.fr/ark:/12148/bpt6k3152t}}

\bibitem[Wie24]{Wie24}
N.~Wiener, \emph{The {Dirichlet} problem.}, J. Math. Phys., Mass. Inst. Techn.
  \textbf{3} (1924), 127--146 (English).
  DOI:\,\href{https://doi.org/10.1002/sapm192433127}{\nolinkurl{10.1002/sapm192433127}}

\end{thebibliography}
\end{document}